\documentclass[a4paper]{article}

\usepackage[top=2.4cm,bottom=3.2cm,left=1.8cm,right=1.8cm]{geometry} 

\usepackage{amsthm,amssymb,amsbsy,amsmath,amsfonts,amssymb,amscd,mathrsfs}
\usepackage{bbold}
\usepackage{enumitem}
\usepackage{graphicx}

\newtheorem{prop}{Proposition}[section]

\newtheorem{rem}{Remark}[section]
\newtheorem{lem}{Lemma}[section]
\newtheorem{assum}{Assumption}

\newcommand{\Rset}{\mathbb{R}}
\newcommand{\one}{\mathbb{1}}
\newcommand{\interior}{\mbox{int}\,}

\DeclareMathOperator*{\esssup}{ess\,sup}

\DeclareMathOperator{\sign}{sign}

\graphicspath{{IMAGES/}{IMAGES/Ex-academic/}{IMAGES/SIR/}}

\title{Equivalent formulations of optimal control problems\\ 
with maximum cost and applications}
\author{{\sc Emilio Molina$^{1,2,3}$, Alain Rapaport $^3$ and Hector Ramirez$^2$}\\[2mm]
$^1$ DIM \& CMM, Santiago-de-Chile, Universidad de Chile\\
$^2$ LJLL, Sorbonne Universite \& INRIA, France\\
$^3$ MISTEA, Univ. Montpellier \& INRAE, France
}

\date{\today}

\begin{document}
	
	\maketitle
	
	\begin{abstract}
	    We revisit the optimal control problem with maximum cost with the objective to provide different equivalent reformulations suitable to numerical methods. We propose two reformulations in terms of extended Mayer problems with constraint, and another one in terms of a differential inclusion with upper-semi continuous right member but without constraint. For this last one we also propose an approximation scheme of the optimal value from below. These approaches are illustrated and discussed on several examples.\\
	    
	    
	\end{abstract}

	\section{Introduction}
	We consider the optimal control problem which consists in minimizing the maximum of a scalar function over a time interval 
	\[
	 \inf_{u(\cdot)}\esssup_{t \in [t_0,T]} y(t)
	\]
	where $y(t)=\theta(t,\xi(t))$ with $\xi(\cdot)$ solution of a controlled dynamics $\dot \xi=\phi(\xi,u)$, $\xi(t_0)=\xi_0$.
	This problem is not in the usual  Mayer, Lagrange or Bolza forms of the optimal control theory, and therefore is not suitable to use the classical necessary optimality conditions of Pontryagin Maximum Principle or existing solving algorithms (based on direct method, shooting or Hamilton-Bellman Jacobi equation). However, this problem falls into the class of optimal control with $L_\infty$ criterion, for which several characterizations of the value function have been proposed  in the literature \cite{BarronIshii,BarronJensenLiu,GonzalezAragone}. Typically, the value function is solution, in a general sense, of a variational inequality of the form
	\[
	\min\Big(\partial_t V+  \inf_u \partial_\xi V.\phi(x,u) \; , \; V-\theta \Big)=0 
	\]
	without boundary condition.
	Nevertheless, although necessary optimality conditions and numerical procedures have been formulated \cite{Barron,DiMarcoGonzalez1996,DiMarcoGonzalez1999,GianattiAragoneLolitoParente}, there is no practical numerical tool to solve such problems as it exists for Mayer problems, to the best of our knowledge. The aim of the present work is to study different reformulations of this problem into Mayer form in higher dimension with possibly state or mixed constraint, for which existing numerical methods can be used.  Indeed, it has already been underlined in the literature that discrete-time optimal control problems with maximum cost do not satisfy the Principle of Optimality but can be transformed into problems of higher dimension with additively separable objective functions \cite{MorganPeet2020,MorganPeet2021}. We pursue here this idea but in the continuous time framework, which faces the lack of differentiability of the max function.
	
	The paper is organized as follows. In Section \ref{sec_prelim}, we give the setup, hypotheses and define the problem. In Section \ref{sec_constraint}, we give equivalences with two Mayer problems with fixed initial condition, under state or mixed constraint.  In Section \ref{sec_noconstraint}, we propose another formulation without constraint in terms of differential inclusion, and then show how the optimal value can be approximated from below by a sequence of more regular Mayer problems.
	Section \ref{sec_illust} is devoted to numerical illustrations. We first consider a very particular class of problems for which we are able to give explicitly the optimal solution, which allows to compare the numerical performances of the different formulations. We then consider a more sophisticated problem from epidemiology, and discuss the various issues in numerical implementations of the different formulations. We also compare numerically with $L_p$ approximations.
	Finally, we discuss in Section \ref{sec_conclusion} about the potential merits of the different formulations as practical methods to compute optimal solution of $L_\infty$ control problems.

	\section{Problem and hypotheses}\label{sec_prelim}
	
	We shall consider autonomous dynamical systems defined on a invariant domain ${\cal D}$ of $\Rset^{n+1}$ of the form
\begin{equation}
\label{sys}
\left\{\begin{array}{l}
	\dot x = f(x,y,u)\\
	\dot y = g(x,y,u)
\end{array}\right.
\end{equation}
(where $g$ is a scalar function) with $u \in U \subset \Rset^p$. Throughout the paper, we shall assume that the following properties are fulfilled.
\begin{assum}$ $
	\label{assum1}
\begin{enumerate}[label=\roman*.]
\item $U$ is a compact set.	
\item The maps $f$ and $g$ are $C^1$ on ${\cal D}\times U$.
\item The maps $f$ and $g$ have linear growth, that is there exists a number $C>0$ such that
\begin{equation*}
	||f(x,y,u)||+|g(x,y,u)|\leq C(1+||x||+|y|), \; (x,y)\in {\cal D}, \; u \in U
\end{equation*}	
\end{enumerate}
\end{assum}

\medskip

For instance, $y(\cdot)$ can be a smooth output of a dynamics
\[
\dot x=f(x,u), \quad y=h(x)
\]
which can be rewritten as
\[
\left\{\begin{array}{l}
\dot x=f(x,u)\\
\dot y =g(x,u):=\nabla h(x)^T\cdot f(x,u)
\end{array}\right.
\]

\medskip

Let ${\cal U}$ be the set of measurable functions $u(\cdot): [0,T] \mapsto U$ and consider $(x_0,y_0)\in {\cal D}$, $T>0$. Under the usual arguments of the theory of ordinary differential equations,  Assumption \ref{assum1} ensures that for any $u(\cdot) \in {\cal U}$ there exists an unique absolutely continuous solution $(x(\cdot),y(\cdot))$ of \eqref{sys} on $[0,T]$ for the initial condition $(x(0),y(0))=(x_0,y_0)$. Define then the solutions set
\[
{\cal S} := \{ (x(\cdot),y(\cdot)) \in {\cal AC}([0,T],\Rset^{n+1}), \mbox{ sol. of \eqref{sys} for } u(\cdot) \in {\cal U} \mbox{ with } (x(0),y(0))=(x_0,y_0)\}
\]
We consider then the optimal control problem which consists in minimizing the "peak" of the function $y(\cdot)$:
\begin{equation*}
{\cal P} : \quad \inf_{u(\cdot) \in {\cal U}} \left ( \max_{t \in [0,T]}y(t)\right)=\inf_{(x(\cdot),y(\cdot)) \in {\cal S}} \left(\max_{t \in [0,T]}y(t)\right)
\end{equation*}

\section{Formulations with constraint}
\label{sec_constraint}
A first approach considers the family of constrained sets of solutions
\[
{\cal S}_z := \{ (x,y) \in {\cal S}, \; y(t) \leq z, \, t \in [0,T] \}, \quad (z \in \Rset)
\]
and to look for the optimization problem
\[
\inf \{ z; \;  {\cal S}_z \neq  \emptyset \}
\]
This problem can be reformulated as a Mayer problem
\begin{equation*}
{\cal P}_0: \inf_{u(\cdot) \in {\cal U}} z(T) 
\end{equation*}
for the extended dynamics in ${\cal D}\times\Rset$
\begin{equation*}
\left\{\begin{array}{l}
\dot x= f(x,y,u)\\
\dot y = g(x,y,u)\\
\dot z = 0
\end{array}\right. 
\end{equation*}
under the state constraint 
\begin{equation*}
	{\cal C}: \quad z(t)-y(t)\geq 0, \; t \in [0,T]
\end{equation*}
where $z(0)$ is free.
Direct methods can be used for such a problem. However, as $z(0)$ is free, solutions are not sought among solutions of a Cauchy problem, which prevents using other methods based on dynamic programming such as the Hamilton-Jacobi-Bellman equation.

\bigskip

We propose another extended dynamics in ${\cal D}\times\Rset$ with an additional control $v(\cdot) \in [0,1]$ 
\begin{equation}
\label{sysext}
\left\{\begin{array}{l}
\dot x= f(x,y,u)\\
\dot y = g(x,y,u)\\
\dot z = \max(g(x,y,u),0)(1-v)
\end{array}\right. 
\end{equation}
Let ${\cal V}$ be the set of measurable functions $v: [0,T] \mapsto [0,1]$.
Note that under Assumption \ref{assum1}, for any $(x_0,y_0,z_0) \in {\cal D}\times \Rset$ and $(u,v) \in {\cal U}\times{\cal V}$, there exists an unique absolutely solution $(x(\cdot),y(\cdot),z(\cdot))$ of \eqref{sysext} on $[0,T]$ for the initial condition $(x(0),y(0),z(0))=(x_0,y_0,z_0)$.
Here, we fix the initial condition with $z_0=y_0$ and consider the Mayer problem
	\begin{equation*}
		{\cal P}_1: \quad \inf_{(u(\cdot),v(\cdot))\in{\cal U}\times{\cal V}} z(T)
		\quad \mbox{under the constraint } {\cal C}
	\end{equation*}
and shows its equivalence with problem ${\cal P}$. We first consider fixed controls $u(\cdot)$.
\begin{prop}
\label{prop0}
	For any control $u(\cdot) \in {\cal U}$, the optimal control problem
	\begin{equation}
	\label{pb_aux}
	\inf_{v \in {\cal V}} z(T) \mbox{ under the constraint } {\cal C}
	\end{equation}
	admits an optimal solution. Moreover,  an optimal solution verifies
	\begin{equation}
	z(T)=\max_{t \in [0,T]}y(t) .
	\end{equation}
	and is reached for a control $v(\cdot)$ that takes values in $\{0,1\}$.
\end{prop}

\begin{proof}
	From equations \eqref{sysext}, one get that any solution $z(\cdot)$ is non decreasing, and as $z$ satisfies the constraint $z\geq y$, we deduce that one has
\begin{equation}
\label{ineqZTbelow}
	z(T)\geq \max_{t \in [0,T]} y(t)
\end{equation}
for any solution of \eqref{sysext}, and thus
\begin{equation*}
\max_{t \in [0,T]} y(t) \leq 	\inf_{v \in {\cal V}} z(T) \mbox{ under the constraint } z(t)\geq y(t), \; t \in [0,T] .
\end{equation*}
Let $x(\cdot)$, $y(\cdot)$ be the solution of \eqref{sys} for the control $u(\cdot)$ and let $I$ be the set of {\em invisible points from the left} of $y$, that is
\[
I:= \left\{ t \in (0,T) ; \; y(t')>y(t) \mbox{ for some } t'<t \right\} .
\]
Consider then the control
\begin{equation}
\label{controlv}
v(t)=\begin{cases}
1, & t \in \interior I,\\
0, & t \notin \interior I
\end{cases}
\end{equation}
When $I$ is empty, $y(\cdot)$ is a non decreasing function, and with the control $v=1$, one has $z(t)=y(t)$ for any $t \in [0,T]$. Therefore one has
\[
z(T)=y(T)=\max_{t \in [0,T]} y(t)
\]
When $I$ is non empty, there exists, from the sun rising Lemma \cite{Tao}, a countable set of disjoint non-empty intervals $I_n=(a_n,b_n)$ of $[0,T]$ such that 
\begin{itemize}
	\item[-] the interior of $I$  is the union of the intervals $I_n$,
	\item[-] one has $y(a_n)=y(b_n)$ if $b_n\neq T$,
	\item[-] if $b_n=T$, then $y(a_n)\geq y(b_n)$.
\end{itemize}
Note that when $t \notin \interior I$, one has $y(t)\geq y(t')$ for any $t'\leq t$. Therefore, the solution $z$ with control \eqref{controlv} verifies
\[
z(t)=\begin{cases}
y(t), & t \notin \interior I\\
y(a_n), & t \in I_n \mbox{ for some } n
\end{cases}
\]
(see Figure \ref{figex} as an illustration).
Let $\bar t \in [0,T]$ be such that
\[
 y(\bar t)=\max_{t \in [0,T]} y(t) ,
\]
which implies that any point $t'>\bar t$ in $[0,T]$ is invisible from the left. Then, one has $z(T)=z(\bar t)\leq y(\bar t)$. With \eqref{ineqZTbelow}, we obtain
\[
\max_{t \in [0,T]} y(t)=z(T)
\]
and deduce
\[
\max_{t \in [0,T]} y(t)=\inf_{v(\cdot)\in{\cal V}} z(T) \mbox{ under the constraint } {\cal C}
\]
\end{proof}

\begin{figure}[ht!]
	\begin{center}
		\includegraphics[width=0.5\textwidth]{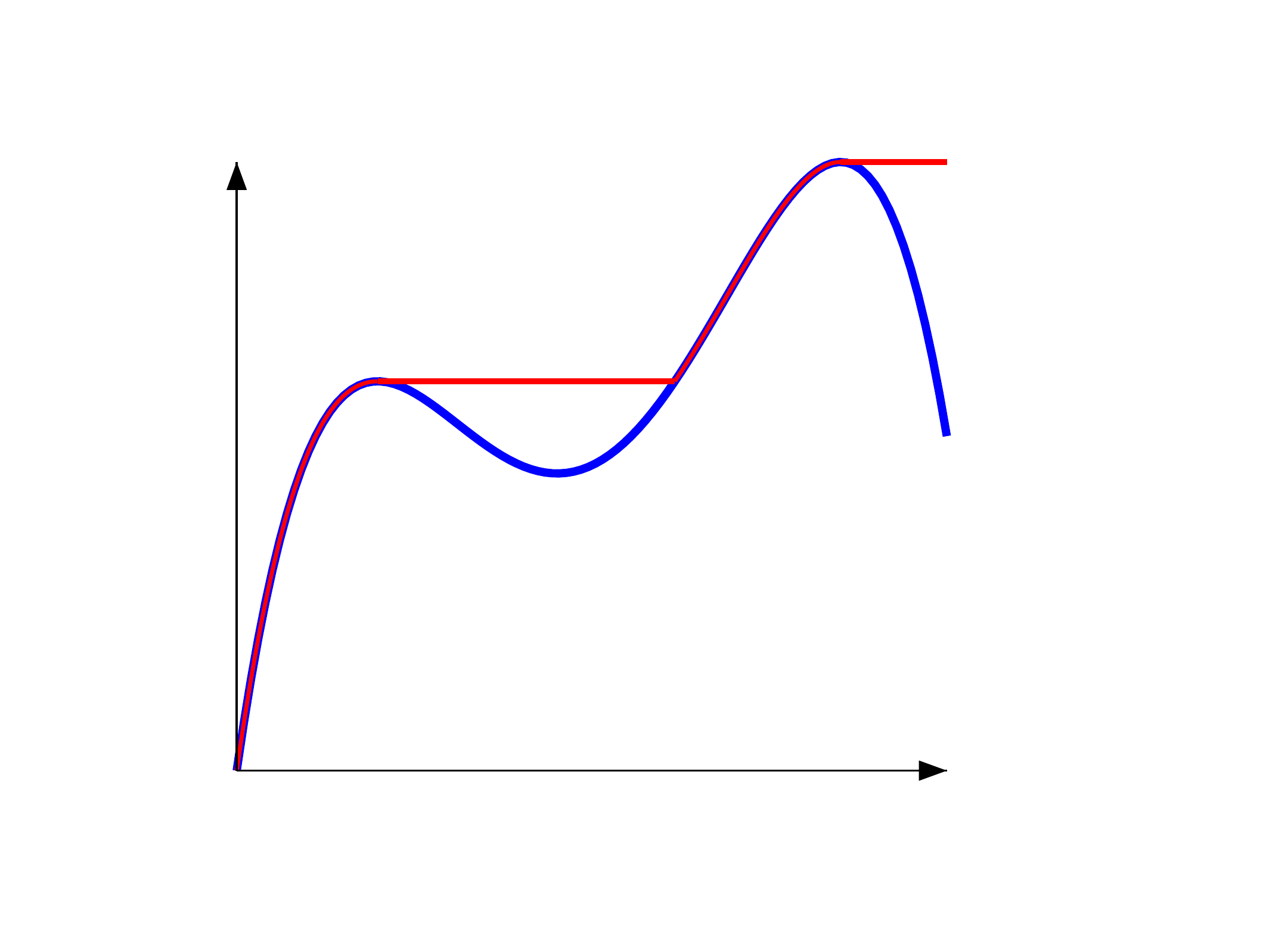}
	\end{center}
\caption{Illustration of the function $z$ (in red) corresponding to a function $y$ (in blue)  with the control given by expression \eqref{controlv}}
\label{figex}
\end{figure}

\begin{rem}
	The proof of Proposition \ref{prop_state_const}  gives an optimal construction of $z(\cdot)$ which is the lower envelope of non decreasing continuous functions above the function $y(\cdot)$, as depicted on Figure \ref{figex}. However, there is no uniqueness of the optimal control $v(\cdot)$.  Any admissible solution $z(\cdot)$ that is above $y(\cdot)$ and such that $z(t)=\hat y$ for $t \geq \hat t=\min \{ t \in (0,T], y(t)=\hat y \}$, where $\hat y:=\max_{s \in [0,T]} y(s)$, is also optimal.
\end{rem}

\medskip

We then obtain the equivalence between problems ${\cal P}_1$ and ${\cal P}$ in the following sense.

\begin{prop}
	\label{prop_state_const}
	If $(u^\star(\cdot),v^\star(\cdot))$ is optimal for Problem ${\cal P}_1$, then $u^\star(\cdot)$ is optimal for Problem ${\cal P}$. Conversely, if $u^\star(\cdot)$ is optimal for Problem ${\cal P}$, then $(u^\star(\cdot),v^\star(\cdot))$ is optimal for Problem ${\cal P}_1$ where $v^\star(\cdot)$ is optimal for the problem \eqref{pb_aux} for the fixed control $u^\star(.)$.
\end{prop}

\bigskip

Let us give another equivalent Mayer problem but with a mixed constraint (this will be useful in the next section).
We consider again the extended dynamics \eqref{sysext} with control $v \in [0,1]$ and the initial condition $(x(0),y(0),z(0))=(x_0,y_0,y_0)$, and define the mixed constraint
\begin{equation*}
{\cal C}_m: \quad \max(y(t)-z(t),0)(1-v(t))+z(t)-y(t)\geq 0, \quad \mbox{a.e. } t \in [0,T]
\end{equation*}
with the optimal control problem
\begin{equation*}
{\cal P}_2: \quad \inf_{(u(\cdot),v(\cdot))\in{\cal U}\times{\cal V}} z(T)
\quad \mbox{under the constraint } {\cal C}_m
\end{equation*}

\begin{prop}
	\label{prop_mixed_cons}
	Problems ${\cal P}_1$ and ${\cal P}_2$ are equivalent.
\end{prop}

\begin{proof}
	One can immediately see that for any admissible solution that satisfies constraint ${\cal C}$, the constraint ${\cal C}_m$ is necessarily fulfilled as $\max(y-z,0)$ is identically null.
	
	Conversely, fix an admissible control $u(\cdot)$ and consider a control $v(\cdot)$ that satisfies ${\cal C}_m$. We show that this implies that the solution $(y(\cdot),z(\cdot))$ verifies necessarily $z(t)\geq y(t)$ for any $t \in [0,T]$. If not, consider the non-empty set
	\begin{equation*}
	E:=\{ t \in [0,T]; \; z(t)-y(t)<0 \} .
	\end{equation*}
	which is open as $z-y$ is continuous. Note that one has $\dot z(t)-\dot y(t)\geq 0$ for a.e.~$t \in E$. Therefore $z-y$ is non decreasing in $E$ and we deduce that for any $t \in E$, the interval $[0,t]$ is necessarily included in $E$, which then contradicts the initial condition $z(0)=y(0)$. 
\end{proof}

\section{Formulation without constraint and approximation}
\label{sec_noconstraint}

We posit $\Pi=(x,y,z) \in {\cal D}\times\Rset$ and consider the differential inclusion
\begin{equation}
\label{diff_incl}
	\dot \Pi \in F(\Pi):= \bigcup_{(u,v) \in U\times[0,1]}  
		\begin{bmatrix}
	f(x,y,u)\\
	g(x,y,u)\\
	h(x,y,z,u,v)
	\end{bmatrix}
\end{equation}
with
\begin{equation*}
h(x,y,z,u,v)=\max(g(x,y,u),0)(1-v\one_{\Rset^+}(z-y))
\end{equation*}
where $\one_{\Rset^+}$ is the indicator function 
\begin{equation*}
\one_{\Rset^+}(\zeta)=\begin{cases}
1, & \zeta \geq 0\\
0, & \zeta <0
\end{cases}
\end{equation*}
Let $\Pi_0=(x_0,y_0,y_0)$ and denote by ${\cal S}_\ell$ the set of absolutely continuous solutions of \eqref{diff_incl} with $\Pi(0)=\Pi_0 \in {\cal D}\times\Rset$. We consider the Mayer problem
\begin{equation*}
{\cal P}_3: \quad \inf_{\Pi(\cdot) \in {\cal S}_\ell}  z(T) .
\end{equation*}

\begin{assum}
	\label{assum_conv}
	\begin{equation*}
		\forall (x,y) \in {\cal D}, \quad
		G(x,y):=\bigcup_ {u \in U}\begin{bmatrix}
			f(x,y,u)\\
			g(x,y,u)
		\end{bmatrix}  \mbox{ is convex},
	\end{equation*}
\end{assum}
\begin{prop}
\label{propP3}
	Under Assumption \ref{assum_conv},
	problem ${\cal P}_3$ admits an optimal solution. Moreover, any optimal solution $\Pi(\cdot)=(x(\cdot),y(\cdot),z(\cdot))$ verifies
	\begin{equation*}
		z(T)=\max_{t \in [0,T]} y(t)
	\end{equation*}
with $(x(\cdot),y(\cdot))$ solution of \eqref{sys} for some control $u(\cdot) \in {\cal U}$ that is optimal for problem ${\cal P}$. 
\end{prop}

\begin{proof}	
	We fix the initial condition $\Pi(0)=\Pi_0$ and consider the augmented dynamics
	\begin{equation}
	\label{diff_incl_ext}
\dot \Pi \in F^\dag(\Pi):= \bigcup_{(u,v,\alpha) \in U\times[0,1]^2}  
\begin{bmatrix}
f(x,y,u)\\
g(x,y,u)\\
	 h^\dag(x,y,z,u,v,\alpha)
\end{bmatrix}
	\end{equation}
	with
	\begin{equation*}
	 h^\dag(x,y,z,u,v,\alpha)=(1-\alpha)h(x,y,z,u,v)+\alpha\max_{w \in U} h(x,y,z,w,0)
	\end{equation*}
	Under Assumption \ref{assum_conv}, the values of $F^\dag$ are convex compact. One can straightforwardly check that the set-valued map $F^\dag$ is upper semi-continuous\footnote{A set-valued map $F: {\cal X} \leadsto {\cal X}$ is upper semi-continuous at $\xi \in {\cal X}$ if and only if for any neighborhood ${\cal N}$ of $F(\xi)$, there exists $\eta>0$ such that for any $\xi' \in B_{\cal X}(\xi,\eta)$ one has $F(\xi') \subset {\cal N}$ (see for instance \cite{Aubin}).} with linear growth. Therefore, the reachable set ${\cal S}_\ell^\dag(T)$ (where ${\cal S}_\ell^\dag$ denotes the set of absolutely continuous solutions of \eqref{diff_incl_ext} with $\Pi(0)=\Pi_0$) is compact (see for instance \cite[Proposition 3.5.5]{Aubin}). Then, there exists a solution $\Pi^\star(\cdot)=(x^\star(\cdot),y^\star(\cdot),z^\star(\cdot))$ of \eqref{diff_incl_ext} which minimizes $z(T)$. \\
	
	Note that any admissible solution $(x(\cdot),y(\cdot),z(\cdot))$ of system \eqref{sysext} that satisfies the constraint ${\cal C}_m$ 
	belongs to ${\cal S}_\ell \subset {\cal S}_\ell^\dag$. We then get the inequality
	\begin{equation}
		\label{ineq1}
	z^\star(T) \leq \inf\{ z(T); \; (x(\cdot),y(\cdot),z(\cdot)) \mbox{ sol. of \eqref{sysext} with } {\cal C}_m \} .
	\end{equation}
	Let us show that any solution $\Pi(\cdot) =(x(\cdot),y(\cdot),z(\cdot))$ in ${\cal S}_\ell$
	verifies
	\begin{equation}
		\label{ineq_z}
	z(T)\geq \max_{t \in [0,T]} y(t)
	\end{equation}
	 We show that one has $z(t)\geq y(t)$ for any $t \in [0,T]$. We proceed by contradiction, as in the proof of Proposition \ref{prop_mixed_cons}. If the set $E=\{ t \in (0,T); \; z(t)-y(t)<0\}$ is non-empty, one has $\dot z(t)-\dot y(t)\geq 0$ for a.e.~$t \in E$ which implies, by continuity, that one has $z(0)-y(0)<0$ which contradicts the initial condition $z(0)=y(0)$. Moreover, as the map $h$ is non-negative, $z(\cdot)$ is non decreasing and we conclude that \eqref{ineq_z} is verified.
	 
	 On another hand, thanks to Assumptions  \ref{assum1} and \ref{assum_conv}, we can apply Filippov’s Lemma to the set-valued map $G$, which asserts that $(x(\cdot),y(\cdot))$ is solution of \eqref{sys} for a certain $u(\cdot) \in {\cal U}$. With\eqref{ineq_z}, we obtain
	\begin{equation}
		\label{ineq2}
	z^\star(T) \geq \max_{t \in [0,T]} y^\star(t) \geq \inf_{u \in {\cal U}} \left\{ \max_{t \in [0,T]} y(t)  ; \; (x(\cdot),y(\cdot)) \mbox{ sol. of \eqref{sys}} \right\} 
	\end{equation}
	where $(x^\star(\cdot),y^\star(\cdot))$ is solution of \eqref{sys} for a certain $u^\star(\cdot) \in {\cal U}$.
	
	Finally, inequalities \eqref{ineq1} and \eqref{ineq2} with Propositions \ref{prop_state_const} and \ref{prop_mixed_cons} show that
	$z^\star(T)$ is reached by a solution of \eqref{sysext} under the constraint ${\cal C}_m$, and that $u^\star(\cdot)$ is optimal for problem ${\cal P}$. We also conclude that the optimal value $z^\star(T)$ is reached by a solution in ${\cal S}_\ell$, which is thus optimal for problem ${\cal P}_3$.
\end{proof}

\begin{rem}
Let us stress that the function $h$ is not continuous, which does not allow to use Filippov's Lemma for the set valued map $F$. This means that one cannot guarantee a priori that an absolutely continuous solution $\Pi(\cdot)=(x(\cdot),y(\cdot),z(\cdot))$ can be synthesized by a measurable control $(u(\cdot),v(\cdot))$. Proposition \ref{propP3} shows that $(x(\cdot),y(\cdot))$ is indeed a solution of system \eqref{sys} for a measurable control $u(\cdot)$, but one cannot guarantee a priori that $z(\cdot)$ can be generated by a measurable control $v(\cdot)$, what does not matter for our purpose.
\end{rem}

We propose now an approximation from below of the optimal cost with a continuous dynamics. In minimization problems, approximations from below of the optimal value are useful to frame the optimal value of the problem, upper bounds being given by any sub-optimal control of problem ${\cal P}_0$, ${\cal P}_1$, ${\cal P}_2$ or ${\cal P}_3$ (provided typically by a numerical scheme). This will be illustrated in Section \ref{sec_illust}.
Let us consider the family of dynamics parameterized by $\theta >0$
\begin{equation}
    \label{dyn_theta}
    \left\{\begin{array}{l}
\dot x = f(x,y,u)\\
\dot y = g(x,y,u)\\
\dot z = h_\theta(x,y,z,u,v)
\end{array}\right.
\end{equation}
with 
\[
h_{\theta}(x,y,z,u,v)=\max(g(x,y,u),0)(1-v\,e^{-\theta \max(y-z,0)})
\]
(where the expression $e^{-\theta \max(y-z,0)}$ plays the role of an approximation of $\one_{\Rset^+}(z-y)$ when $\theta$ tends to $+\infty$). We then define the family of Mayer problems
\begin{equation*}
	{\cal P}^{\theta}_3: \quad \inf_{\Pi(\cdot) \in {\cal S}_{\theta}} z(T)
\end{equation*}
where ${\cal S}_\theta$ denotes the set of absolutely continuous solutions $\Pi(\cdot)=(x(\cdot),y(\cdot),z(\cdot))$ of \eqref{dyn_theta} for the initial condition $\Pi(0)=\Pi_0$. Let us underline that for these problems with Lipschitz dynamics without constraint, necessary conditions based on Pontryagin Maximum Principle can be derived, leading to shooting methods that are known to very accurate and that could be initialized from numerical solutions of problems ${\cal P}_1$  or ${\cal P}_2$ obtained for instance with direct methods.

\begin{prop} Under Assumption \ref{assum_conv}, for any increasing sequence of numbers $\theta_n$ $(n \in \mathbb{N})$ that tends to $+\infty$, the problem ${\cal P}_3^{\theta_n}$ admits an optimal solution, and for any sequence of optimal solutions  $(x_n(\cdot),y_n(\cdot),z_n)(\cdot))$ of ${\cal P}_3^{\theta_n}$, the sequence $(x_n(\cdot),y_n(\cdot))$ converges, up to  sub-sequence, uniformly to an optimal solution $(x^\star(\cdot),y^\star(\cdot))$ of Problem ${\cal P}$, and its derivatives weakly to $(\dot x^\star(\cdot),\dot y^\star(\cdot))$ in $L_2$. 
Moreover, $z_n(T)$ is an increasing sequence that converges to $\max_{t \in [0,T]} y^\star(t)$.
\end{prop}

\begin{proof}
As in the proof of Proposition \ref{propP3}, we consider for any $\theta>0$ the convexified dynamics 
\[
\left\{\begin{array}{l}
\dot x = f(x,y,u)\\
\dot y = g(x,y,u)\\
\dot z = h_\theta^\dag(x,y,z,u,v,\alpha):=(1-\alpha)h_\theta(x,y,z,u,v)+\alpha\max_{w \in U} h_\theta(x,y,z,w,0)
\end{array}\right.
\]
where $\alpha \in [0,1]$. Then, there exists an absolutely continuous solution $(x^\star_\theta(\cdot),y^\star_\theta(\cdot),z^\star_\theta(\cdot))$ with a measurable control $(u^\star_\theta(\cdot),v^\star_\theta(\cdot),\alpha^\star_\theta(\cdot))$ which minimizes $z(T)$. For the control $(u^\star_\theta(\cdot),v^\star_\theta(\cdot),0)$, the solution is given by $(x^\star_\theta(\cdot),y^\star_\theta(\cdot),\tilde z^\star_\theta(\cdot))$ where $\tilde z^\star_\theta(\cdot)$ is solution of the Cauchy problem
\[
\dot z = \tilde l_\theta(t,z):=h_\theta^\dag(x^\star_\theta(t),y^\star_\theta(t),z;u^\star_\theta(t),v^\star_\theta(t),0), \; z(0)=y(0)
\]
while $z^\star_\theta(\cdot)$ is solution of
\[
\dot z = l_\theta(t,z):=h_\theta^\dag(x^\star_\theta(t),y^\star_\theta(t),z,u^\star_\theta(t),v^\star_\theta(t),\alpha^\star_\theta(t)), \; z(0)=y(0)
\]
One can check that the inequality
\[
\tilde l_\theta(t,z) \leq l_\theta(t,z), \quad t \in [0,T], \; z \in \Rset
\]
is fulfilled, 
which gives by comparison of solutions of scalar ordinary differential equations (see for instance \cite{Walter}) the inequality
\[
\tilde z^\star_\theta(t) \leq z^\star_\theta(t), \quad t \in [0,T]
\]
We deduce that $(x^\star_\theta(\cdot),y^\star_\theta(\cdot),z^\star_\theta(\cdot))$ is necessarily a solution of \eqref{dyn_theta}. 

\bigskip

Let
\[
\bar y:= \inf_{u \in {\cal U}} \left\{ \max_{t \in [0,T]} y(t)  ; \; (x(\cdot),y(\cdot)) \mbox{ sol. of \eqref{sys}} \right\} 
\]
By Proposition \ref{propP3}, we know that there exists an optimal solution $(x(\cdot),y(\cdot),z(\cdot))$ of problem ${\cal P}_3$ such that
$z(T)=\bar y$.
Clearly, this solution belongs to ${\cal S}_\theta$ for any $\theta$, and we thus get
\begin{equation}
    \label{upperboundzstartheta}
    z^\star_\theta (T) \leq \bar y
\end{equation}

\medskip

Let 
\[
F_{\theta}(\Pi):= \bigcup_{(u,v) \in U\times[0,1]}  
	\begin{bmatrix}
		f(x,y,u)\\
		g(x,y,u)\\
		h_{\theta}(x,y,z,u,v)
	\end{bmatrix} 
\]
and note that one has
\begin{equation}
    \label{convergenceFtheta}
\lim_{\theta \to +\infty} d\left(F_\theta(\Pi),F(\Pi)\right)=0, \quad \Pi \in {\cal D}\times\Rset
\end{equation}
Consider an increasing sequence of numbers $\theta_n$ ($n \in \mathbb{N}$), and denote $\Pi_n(\cdot)=(x_n(\cdot),y_n(\cdot),z_n(\cdot))$ an optimal solution of problem ${\cal P}_3^{\theta_n}$. Note that one has
\begin{equation}
    \label{Sinclusion}
{\cal S}_{\theta_{n+1}} \subset {\cal S}_{\theta_{n}} \cdots \subset {\cal S}_{\theta_{0}}
\end{equation}
Therefore, the sequence $\dot \Pi_n(\cdot)$ is bounded, and $\Pi_n(\cdot)$ as well. As $F$ is upper semi-continuous, we obtain that $\Pi_n(\cdot)$ converges uniformly on $[0,T]$, up to a sub-sequence, to a certain $\Pi^\star(\cdot)=(x^\star(\cdot),y^\star(\cdot),z^\star(\cdot))$ which belongs to ${\cal S}_l$ (see for instance \cite[Th.~3.1.7]{Clarke}).
From property\eqref{Sinclusion}, we obtain that $z_n(T)$ is a non decreasing sequence that converges to $z^\star(T)$, and from \eqref{upperboundzstartheta}, we get passing at the limit
\[
z^\star(T) \leq \bar y
\]
On another hand, $(x^\star(\cdot),y^\star(\cdot),z^\star(\cdot))$ belongs to ${\cal S}_l$ and we get from Proposition \ref{propP3} the inequality
\[
z^\star(T) \geq \bar y
\]
Therefore, one has $z^\star(T) = \bar y$ and $(x^\star(\cdot),y^\star(\cdot),z^\star(\cdot))$ is then an optimal solution of problem ${\cal P}_3$. From Proposition \ref{propP3}, we obtain that one has necessarily
\[
z^\star(T)=\max_{t \in [0,T]} y^\star(t)
\]
Finally, the sequence $(\dot x_n(\cdot),\dot y_n(\cdot))$ being bounded, it converges, up to a sub-sequence, weakly to $(\dot x^\star(\cdot),\dot y^\star(\cdot))$ in $L_2$ tanks to Alaoglu's Theorem.
\end{proof}

\section{Numerical illustrations}
\label{sec_illust}

We begin by illustrating the different formulations on a problem for which the optimal solution is known.

\subsection{A particular class of dynamics}
	
We consider dynamics of the form
\begin{equation*}
	(\Sigma):
	\left\{\begin{array}{l}
		\dot x= f(x)\\
		\dot y = g(x,u)
	\end{array}\right. \qquad x \in \Rset^n, \; u \in U
\end{equation*}

\begin{prop}\label{particular}
A feedback control $x\mapsto \phi^\star(x)$ such that
\begin{equation*}
	g(x,\phi^\star(x))=\min_{u \in U} g(x,u), \quad x \in \Rset^n
	\end{equation*}	
\end{prop}
is optimal for problem ${\cal P}$.

\begin{proof}
	For a given $x_0$ in $\Rset^n$, let $x(\cdot)$ be the solution of $\dot x=f(x)$, $x(0)=x_0$ independently to the control $u(\cdot)$. Then, for any solution $y(\cdot)$, one has
	\[
	y(t)= y(0)+ \int_0^t g(x(\tau),u(\tau))\,d\tau \geq y(0)+ \int_0^t \min_{v \in U} g(x(\tau),v)\,d\tau , \quad t \geq 0
	\]
	Let $y^\star(\cdot)$ be defined as
	\[
	y^\star(t):= y(0)+ \int_0^t \min_{v \in U} g(x(\tau),v)\,d\tau , \quad t \geq 0
	\]
	Clearly, one has
	\[
	\max_t y(t) \geq \max_t y^\star(t)
	\]
	where $y^\star(\cdot)$ is a solution of $\Sigma$ for any measurable control $u^\star(\cdot)$ such that
	\[
	g(x(t),u^\star(t))=\min_{v \in V} g(x(t),v), \quad \mbox{a.e. } t \geq 0
	\]
	We conclude that $y^\star(\cdot)$ is an optimal trajectory of problem ${\cal P}$ for the control generated by the feedback $\phi^\star$.
\end{proof}

\medskip

As a toy example, we have considered the system
\begin{equation*}
	\left\{\begin{array}{l}
		\dot x=1, \; x(0)=0\\
		\dot y =(1-x)(2-x)(4-x)(1+u/2), \; y(0)=0  
	\end{array}\right. \qquad u \in [-1,1]
\end{equation*}
for which
\begin{equation*}
\phi^\star(x)=-\sign\Big( (1-x)(2-x)(4-x) \Big)
\end{equation*}
is an optimal control which minimizes $\max_{t \in [0,T]} y(t)$.
Remark that this problem can be equivalently written with a scalar non-autonomous dynamics
\[
\dot y = (1-t)(2-t)(4-t)(1+u/2)
\]
for which the open-loop control
\[
u^\star(t)=-\sign\Big( (1-t)(2-t)(4-t) \Big)
\]
is optimal.\\

For $T=5$, we have first computed the exact optimal solution of problem ${\cal P}$ with the open-loop $u^\star(\cdot)$, by integrating the dynamics with Scipy in Python software (see Figure \ref{figsolexacte}).
\begin{figure}[ht!]
	\begin{center}
		\includegraphics[width=0.45\textwidth]{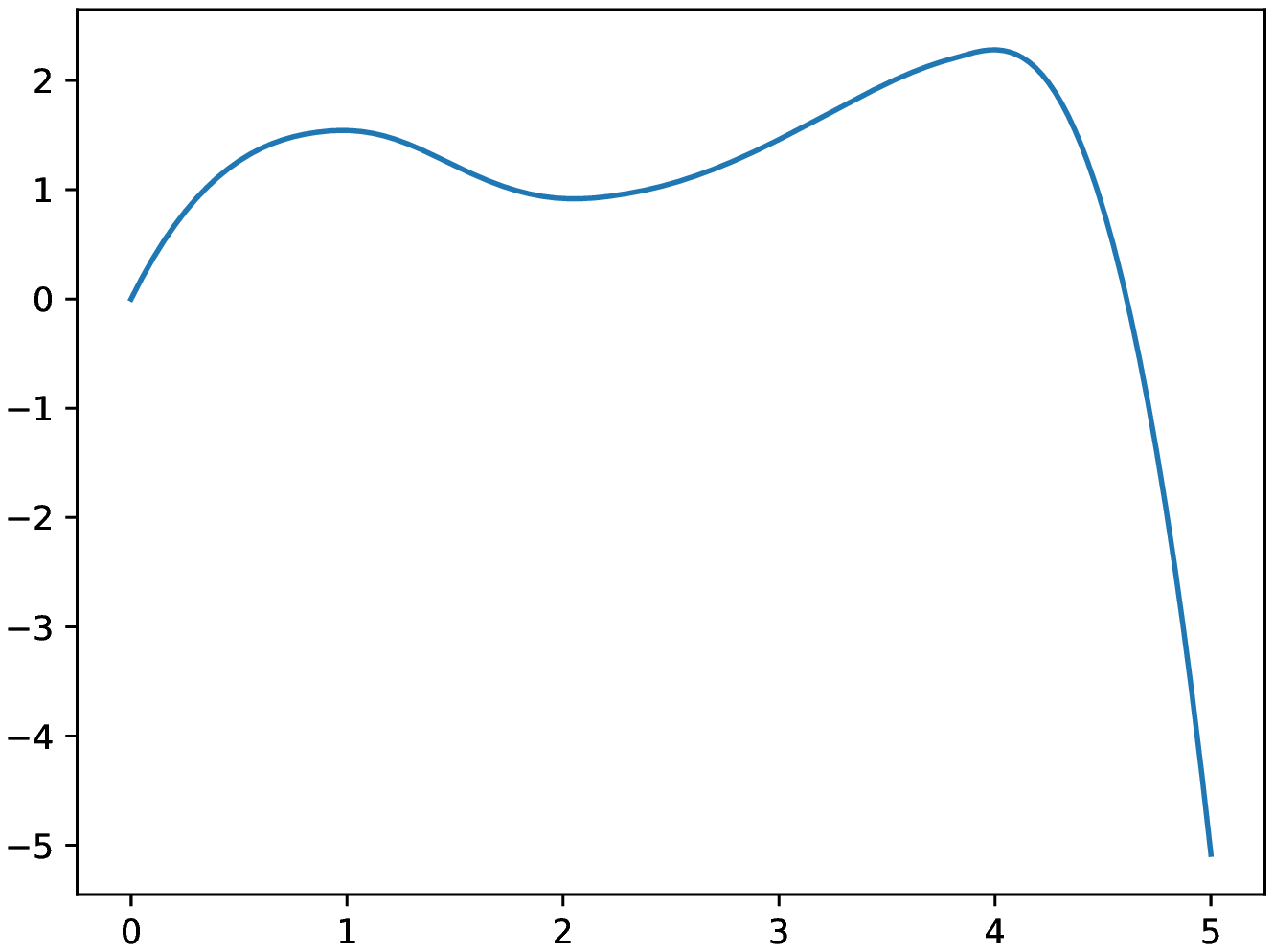}
		\includegraphics[width=0.45\textwidth]{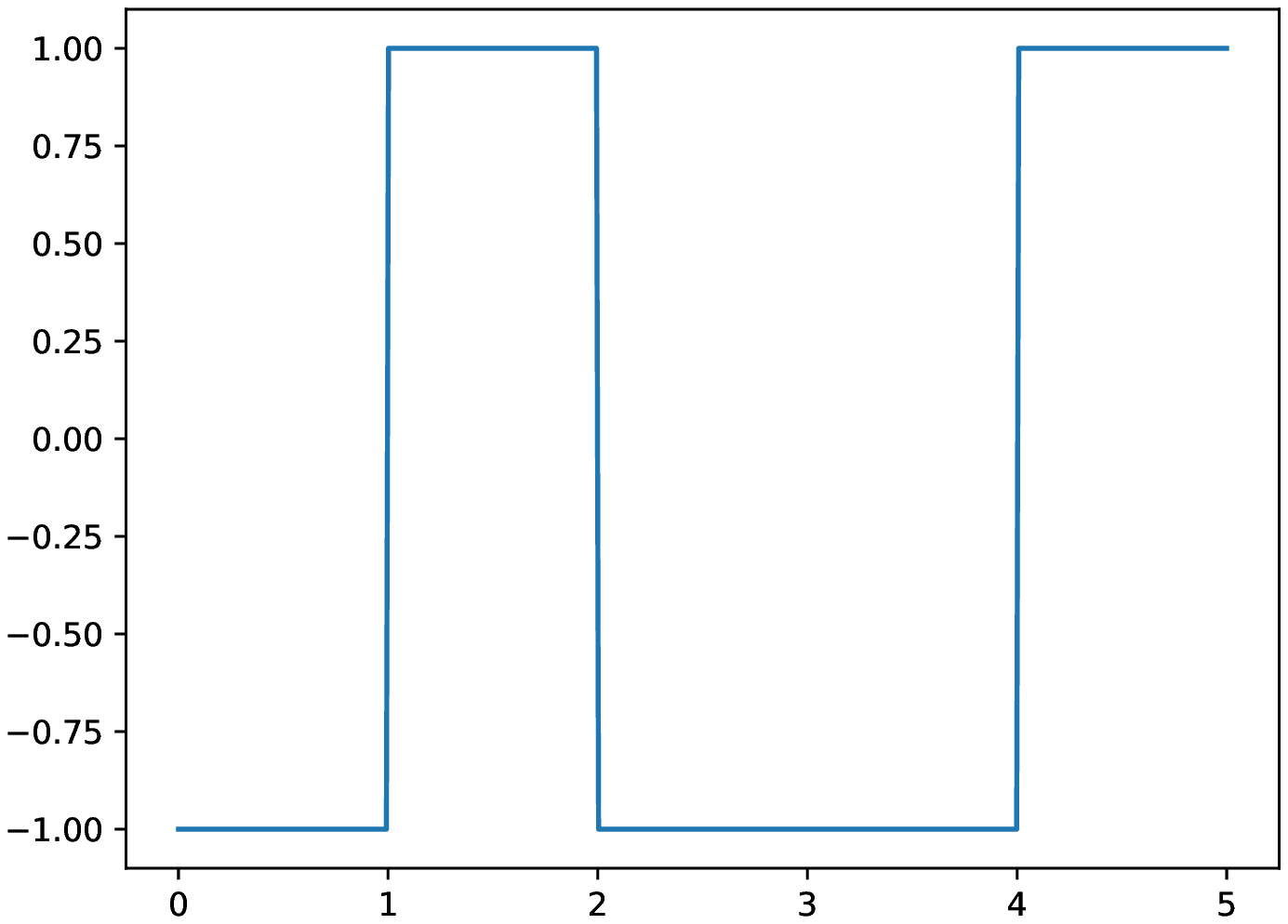}
	\end{center}
	\caption{Optimal solution: $y(\cdot)$ on the left, $u^\star(\cdot)$ on the right}
	\label{figsolexacte}
\end{figure}
Impacts of perturbations on the switching times on the criterion are presented in Table \ref{tablesens}, which show a quite high sensitivity of the optimal control for this problem.
\begin{table}
	\begin{center}
		\begin{tabular}{c|c|c}
			disturbance & $\displaystyle \max_{t\in[0,T]} y(t)$ & error \\ \hline \hline
			$0$ &  $2.24985$ & $0$ \\
			$0.001\%$ & $2.24985$ & $4.10^{-6}\%$ \\
			$0.01\%$ &  $2.25010$ & $0.01\%$ \\
		    $0.1\%$ & $2.69457$	& $20\%$
		\end{tabular}
		\caption{Sensitivity to the optimal switching}
		\label{tablesens}
	\end{center}
\end{table}
Then, we have solved numerically problems ${\cal P}_0$ to ${\cal P}_2$ with a direct method (Bocop software using Gauss II integration scheme) for $500$ time steps and an optimization relative tolerance equal to $10^{-10}$. For problem ${\cal P}_3$, as the dynamics is not continuous, direct methods do not work well and we have used instead a numerical scheme based on dynamic programming (BocopHJB software) with $500$ time steps and a discretization of $200\times 200$ points of the state space. For the additional control $v$, we have considered only two possible values $0$ and $1$ as we know that the optimal solution is reached for $v \in \{0,1\}$ (see Proposition \ref{prop0}).
The numerical results and computation times are summarized in Table \ref{tableresults}, while Figure \ref{figcompar} presents the corresponding trajectories.

\begin{table}
	\begin{center}
	\begin{tabular}{c|c|c|c}
		problem & $\displaystyle \max_{t\in[0,T]} y(t)$ & error & computation time \\ \hline\hline
		${\cal P}$ &  $2.24705$ & 0 & $-$\\ 
		${\cal P}_0$ &  $2.249888$ & $0.126\%$ & $0.5\, s$\\
		${\cal P}_1$ &  $2.24998$ & $0.130\%$ & $1.8\, s$\\
		${\cal P}_2$ &  $2.249941$ & $0.129\%$ & $3.8\, s$\\
			${\cal P}_3$ &  $2.26778$ & $0.8\%$ & $248\, s$
		
	\end{tabular}
\caption{Comparison of the numerical results}
	\label{tableresults}
	\end{center}
\end{table}

We note that the direct method give very accurate results, and the computation time for problem ${\cal P}_0$ is the lowest because it has only one control. The computation time for problem ${\cal P}_2$ is slightly higher than for ${\cal P}_1$ because the mixed constraint ${\cal C}_m$ is heavier to evaluate. The numerical method for problem ${\cal P}_3$ is of completely different nature as it computes the optimal solution for all the initial conditions on the grid, which explains a much longer computation time. The accuracy of the results is also directly related to the size of the discretization grid and can be improved by increasing this size but at the price of a longer computation time.
\begin{figure}[ht!]
	\begin{center}
		\includegraphics[width=0.45\textwidth]{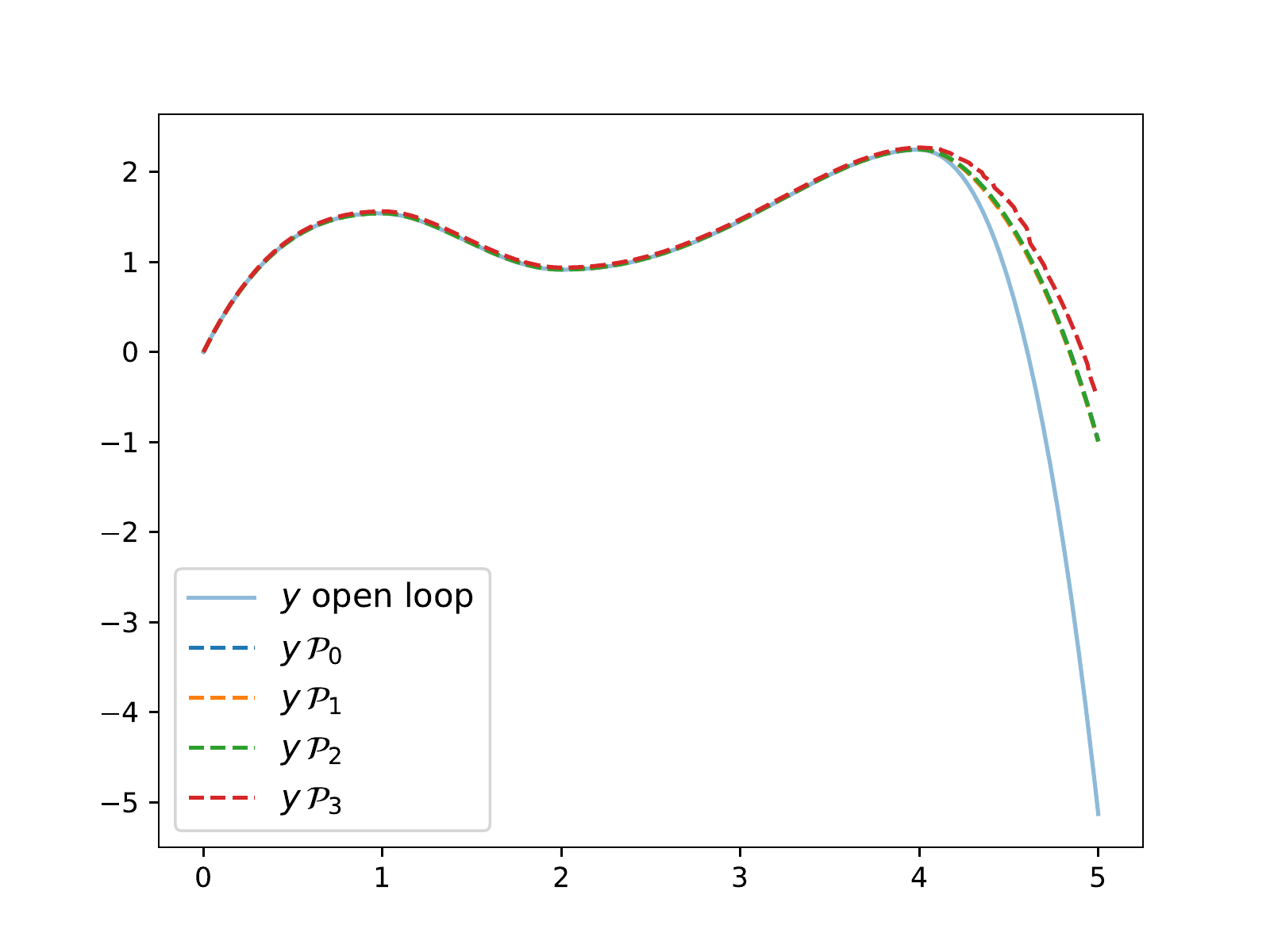}
		\includegraphics[width=0.45\textwidth]{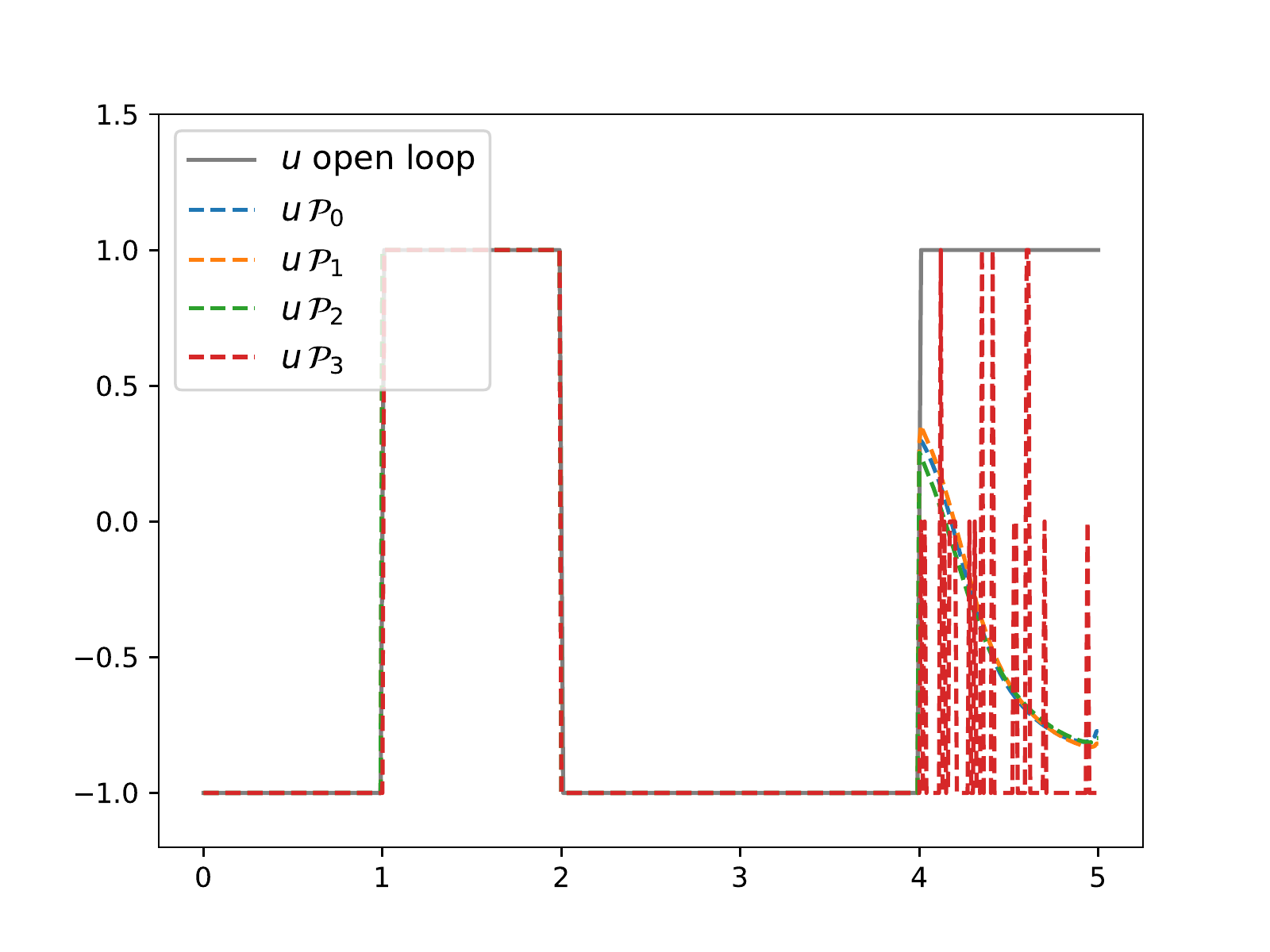}
		\includegraphics[width=0.45\textwidth]{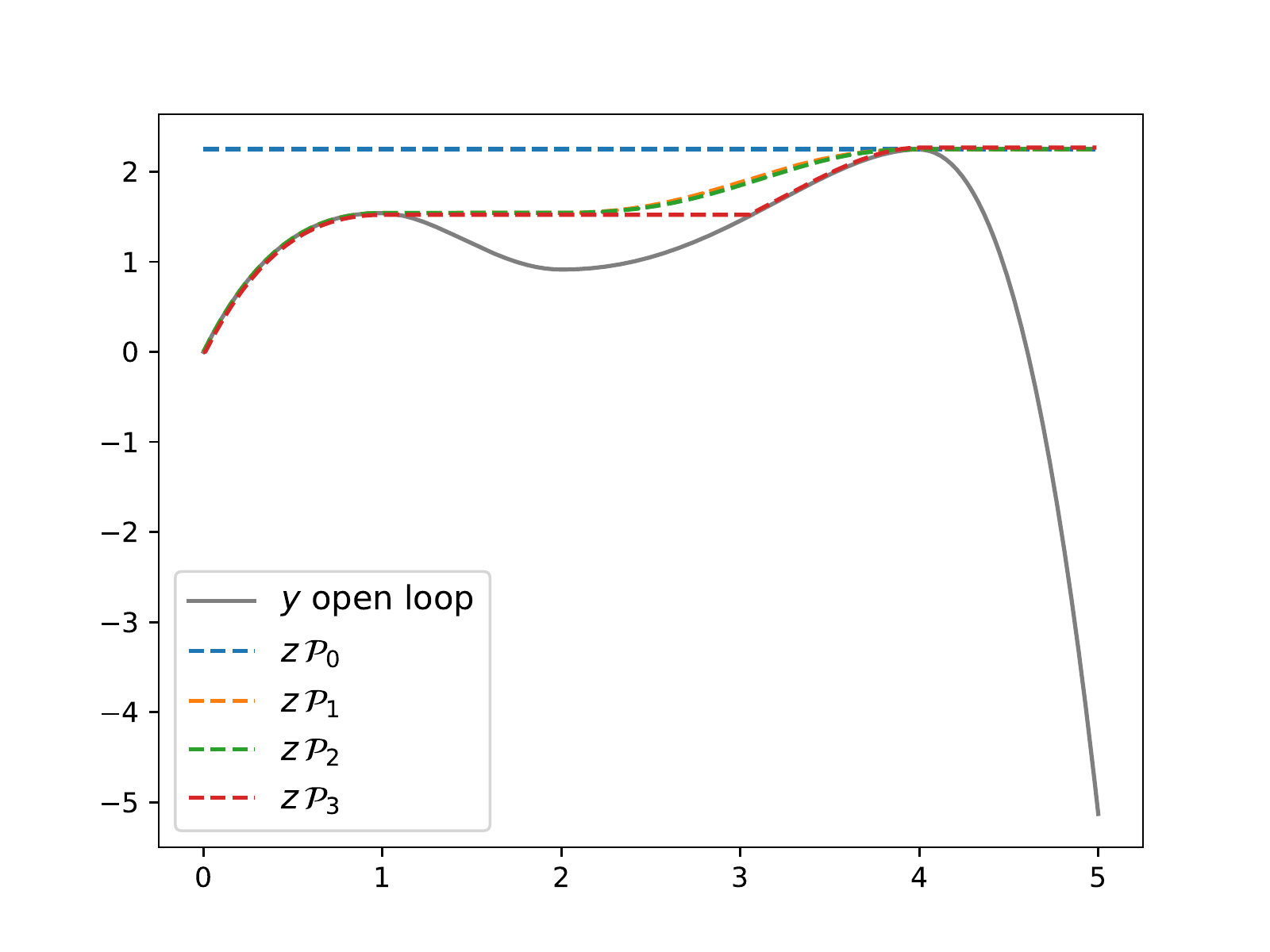}
		\includegraphics[width=0.45\textwidth]{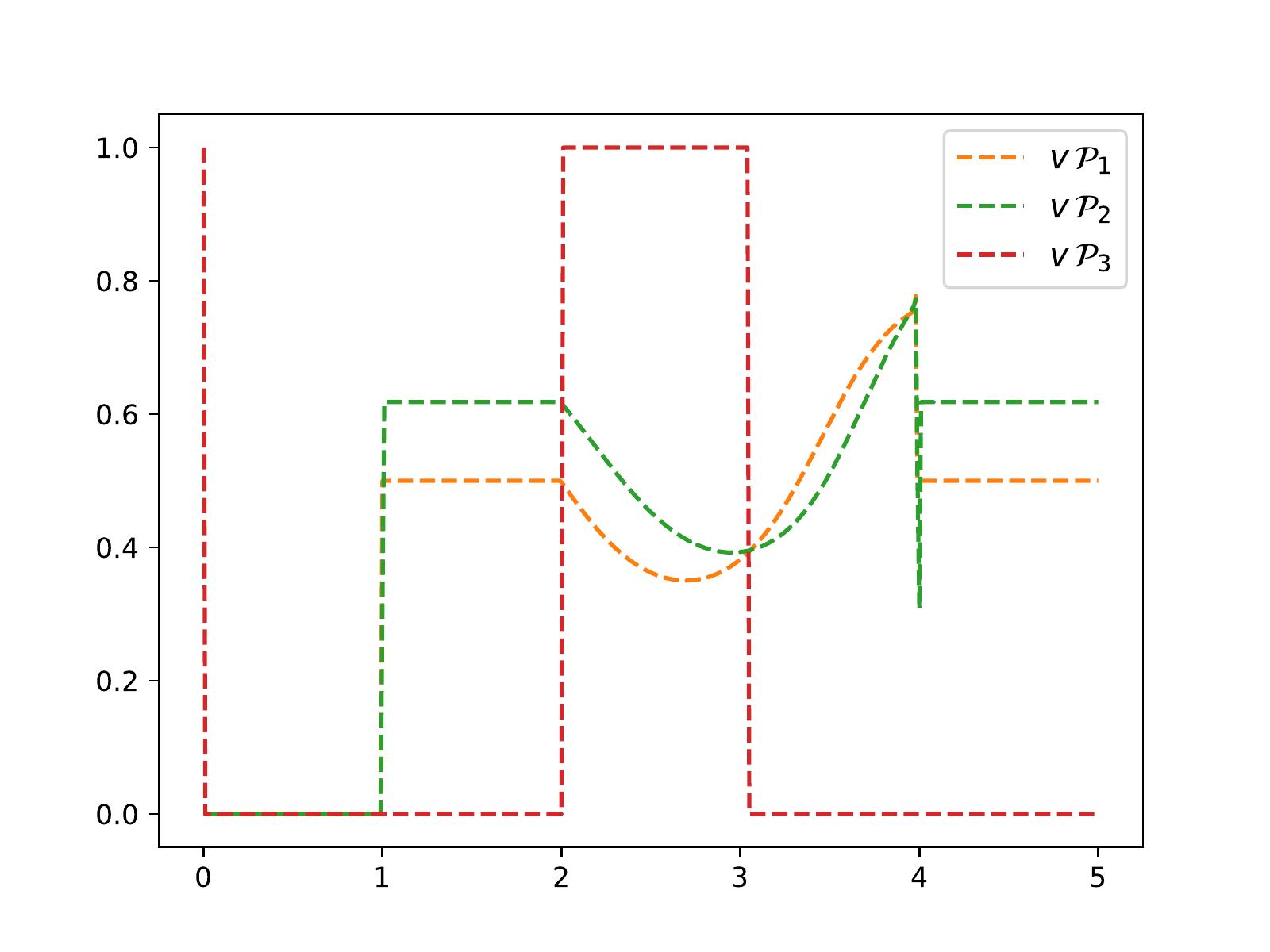}
	\end{center}
	\caption{Comparisons of the three methods on $y(\cdot)$, $u(\cdot)$, $z(\cdot)$ and $v(\cdot)$.}
	\label{figcompar}
\end{figure}

\medskip

On Figure \ref{figcompar}, one may notice some difference between the obtained trajectories. Let us underline that after the peak of $y(\cdot)$, there is no longer uniqueness of the optimal control.

\subsection{Application to an epidemiological model}

The SIR model is one of the most basic transmission model in epidemiology for a directly transmitted infectious disease (for a complete introduction, see for instance \cite{SIR}) and it retakes great importance nowadays due to covid-19 epidemic. 

Consider on a time horizon $[0,T]$ variables $S(t)$, $I(t)$ and $R(t)$ representing the fraction of susceptible, infected and recovery individuals at time $t\in[0,T]$, so that one has $S(t)+I(t)+R(t)=1$ with $S(t),I(t),R(t)\ge 0$. Let $\beta > 0$ be the rate of transmission and $\gamma >0$ the recovery rate. Interventions as lock-downs and curfew are modeled as a factor in rate transmission that we denote $u$ and which represents our control variable taking values in $[0,u_{max}]$ with $u_{max} \in (0,1)$, where $u=0$ means no intervention and $u=u_{max}$ the most restrictive one which reduces as much as possible contacts among population. The SIR dynamics including the control is then given by the following equations:
\begin{align}
	& \dot{S}=-(1-u)\beta SI\\
	& \dot{I}=\,(1-u)\beta SI-\gamma I\\
	& \dot{R}=\gamma I
\end{align}
When the reproduction number ${\cal R}_0=\beta/\gamma$ is above one and the initial proportion of susceptible is above the herd immunity threshold ${\cal R}_0^{-1}$, it is well known that there is an epidemic outbreak. Then, the objective is to minimize the peak of the infected population
\[
\max_{t \in [0,T]} I(t)
\]
with respect to control $u(\cdot)$ subject to a $L_1$ budget
\begin{equation}
\label{budget}
	\int_{0}^{T}u(t)\le Q
\end{equation}
on a given time interval $[0,T]$ where $T$ is in general chosen large enough to ensure the herd immunity of the population is reached at date $T$. Note that one can drop the $R$ dynamics to study this problem. If the constraint \eqref{budget} were not imposed, then the optimal solution would be the trivial control $u(t)=u_{max}, \; t\in[0,T]$, which is in general unrealistic from a operational point of view. A similar problem has been considered in \cite{Morris} but under the constraint that intervention occurs only once on a time interval of given length, that we relax here. Note that the constraint \eqref{budget} can be reformulated as a target condition, considering the augmented dynamics
\begin{align}
	& \dot{S}=-(1-u)\beta SI\\
	& \dot{I}=\,(1-u)\beta SI-\gamma I\\
	& \dot{C}=-u(t)
\end{align}
with initial condition $C(0)=Q$ and target $\{ C \geq 0 \}$.
Extension of the results of Sections \ref{sec_constraint} and \ref{sec_noconstraint} to problems with target do not present any particular difficulty, and is left to the reader.

The parameters considered for the numerical simulations are given in Table \ref{tablepar}. 
\begin{table}[ht!]
	\begin{center}
		\begin{tabular}{l|l|l|l|l|l}
			$\beta$ &  $\gamma$ & $T$ &  $Q$ & $S(0)$ & $I(0)$\\
			\hline\hline
			$0.21$  & $0.07$ &  $300$ & $28$ & $1-10^{-6}$ & $10^{-6}$
		\end{tabular}
		\caption{SIR parameters considered in numerical computations}
		\label{tablepar}
	\end{center}
\end{table} 
 Adding the $z$-variable, we end up with a dynamics in dimension four, which is numerically heavier than for the previous example. In particular, methods based on the value function are too time consuming to obtain accurate results for refined grids in a reasonable computation time. So we have considered direct methods only. We do not consider here problem ${\cal P}_3$, but instead its regular approximations ${\cal P}_3^\theta$ suitable to direct methods. For direct methods that use algebraic differentiation of the dynamics, convergence and accuracy are much better if one provides differentiable dynamics.
This is why we have approximated the $\max(\cdot,0)$ operator for problems ${\cal P}_1$ and ${\cal P}_2$ by the Laplace formula
\[
\dfrac{\log\left(e^{\lambda \xi}+1\right)}{\lambda} \underset{\lambda \to +\infty}{\longrightarrow} \max(\xi,0), \quad \xi \in \Rset
\]
with $\lambda=100$ for the numerical experiments.
For problem ${\cal P}_3^\theta$, one has to be careful about the interplay between the approximations of $\max(\cdot,0)$ and the sequence $\theta_n \to +\infty$, to provide approximations from below of the optimal value.
The function $h_{\theta}$ is thus approximated by the expression
\begin{equation*}
h_\theta(x,y,z,u,v) \simeq			\dfrac{\log\left(e^{\lambda_1 g(x,y,u)}+1\right)}{\lambda_1}\left(1-v e^{\frac{\theta}{\lambda_2} \log\left(e^{\lambda_2 (y-z)}+1\right)}\right)
\end{equation*}
which depends on three parameters $\lambda_1$, $\lambda_2$ and $\theta$.
Posit for convenience
\[
\alpha:=\frac{\theta}{\lambda_2}
\]
and consider the function
$$
\omega_{\alpha,\lambda_2}(\xi):=e^{-\alpha \log\left(e^{-\lambda_2 \xi}+1\right)}, \quad \xi \in \Rset
$$
which approximates the indicator function $\one_{\Rset^+}$. 
One has the following properties.

\begin{lem}\label{lem:f}
$ $
\begin{enumerate}
    \item For any positive numbers $\alpha$, $\lambda_2$, the function $\omega_{\alpha,\lambda_2}$ is increasing with
    \[
    \lim_{\xi\rightarrow-\infty} \omega_{\alpha,\lambda_2}(\xi)=0, \quad  \lim_{\xi\rightarrow+\infty} \omega_{\alpha,\lambda_2}(\xi)=1
    \]
    \item For any $\varepsilon\in(0,1)$, one has $\omega_{\alpha,\lambda_2}\left(-\varepsilon^2\right)=\varepsilon$ and $\omega_{\alpha,\lambda_2}(0)=1-\varepsilon$ exactly  for
    \begin{equation}
        \label{alphalambda2}
    \alpha=-\frac{\log(1-\varepsilon)}{\log(2)}, \quad  \lambda_2=\frac{\log(\varepsilon^{-\frac{1}{\alpha}}-1)}{\varepsilon^2}
    \end{equation}
\end{enumerate}
\end{lem}

\begin{proof}
One has first
$$
\omega'_{\alpha,\lambda_2}(\xi)=\lambda_2 \alpha\dfrac{e^{-\lambda_2 x}}{e^{-\lambda_2\xi}+1}\omega_{\alpha,\lambda_2}(\xi)>0
$$ 
and the function $\omega_{\alpha,\lambda_2}(\cdot)$ is thus increasing. From
$$
\lim_{\xi\rightarrow -\infty} -\alpha\log(e^{-\lambda_2 \xi}+1)= -\infty
$$
one get
$$
\lim_{\xi\rightarrow-\infty} \omega_{\alpha,\lambda_2}(\xi)=0
$$
and similarly
$$
\lim_{\xi\rightarrow +\infty}-\alpha\log(e^{-\lambda_2 \xi}+1)=0
$$
implies
$$
\lim_{\xi\rightarrow+\infty} \omega_{\alpha,\lambda_2}(\xi)=1
$$
Finally, with simple algebraic manipulation of the conditions $\omega_{\alpha,\lambda_2}\left(-\varepsilon^2\right)=\varepsilon$ and $\omega_{\alpha,\lambda_2}(0)=1-\varepsilon$, one obtains straightforwardly the expressions \eqref{alphalambda2}.
\end{proof}
We have taken $\lambda_1=5000$ and considered a sequence of approximations of the indicator function for the values given in Table \ref{tab:alphalambda} according to expressions \eqref{alphalambda2} of Lemma \ref{lem:f} (see Figure \ref{fig:ind}).
\begin{table}[ht!]
    \centering
    \begin{tabular}{l|l|l}
       $\varepsilon$  &  $\alpha$ & $\lambda_2$\\
       \hline\hline
        0.2     & 0.32 & 124\\
        0.15 & 0.234 & 360\\
        0.1 & 0.152 & 1514\\
        0.075 & 0.112 & 4094\\
        0.05 & 0.074 & 16193
    \end{tabular}
    \caption{Values of parameters $\alpha$, $\lambda_2$ for different $\varepsilon$}
    \label{tab:alphalambda}
\end{table}
\begin{figure}[ht!]
	\begin{center}
		\includegraphics[width=0.45\textwidth]{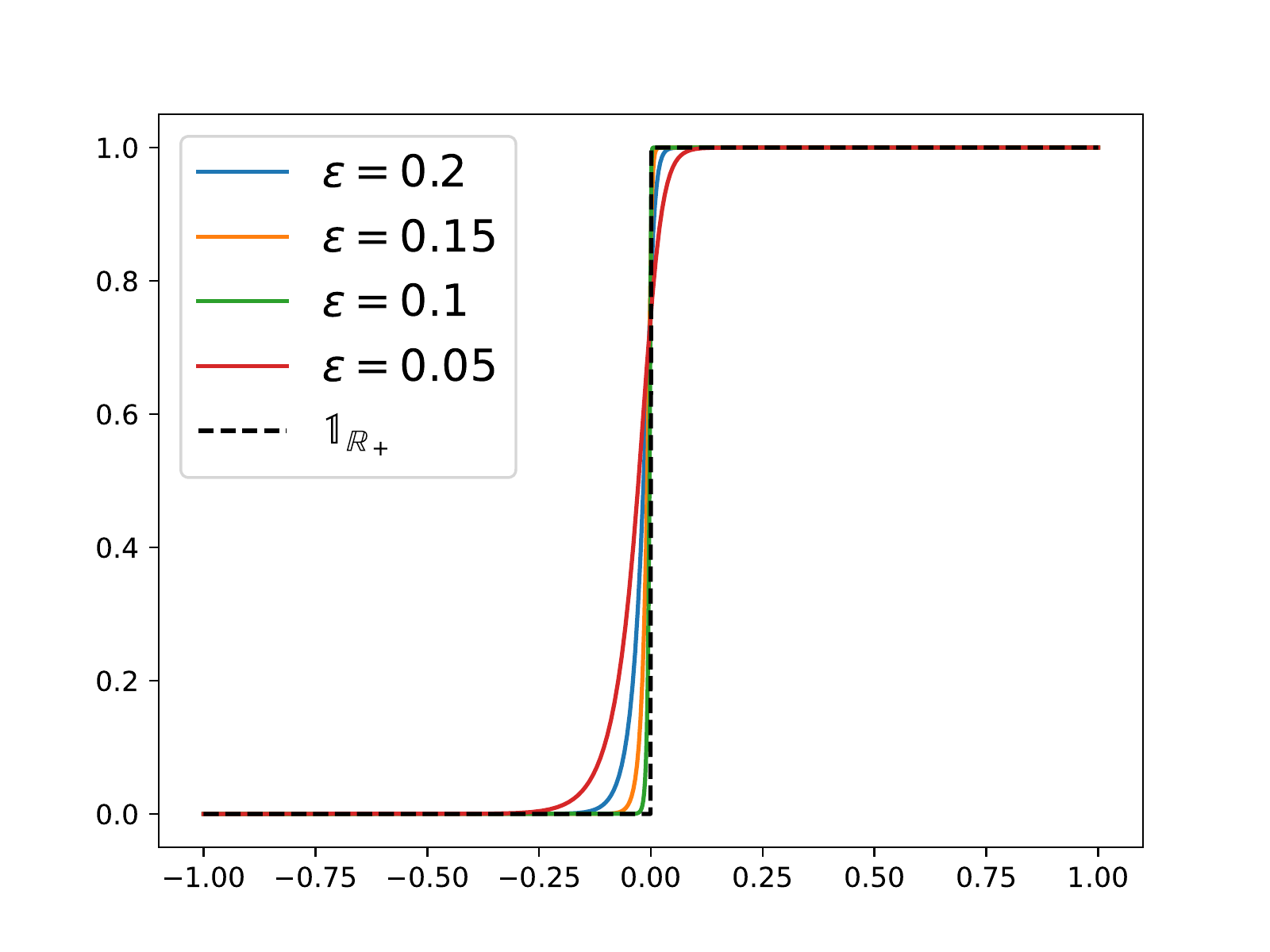}
		\includegraphics[width=0.45\textwidth]{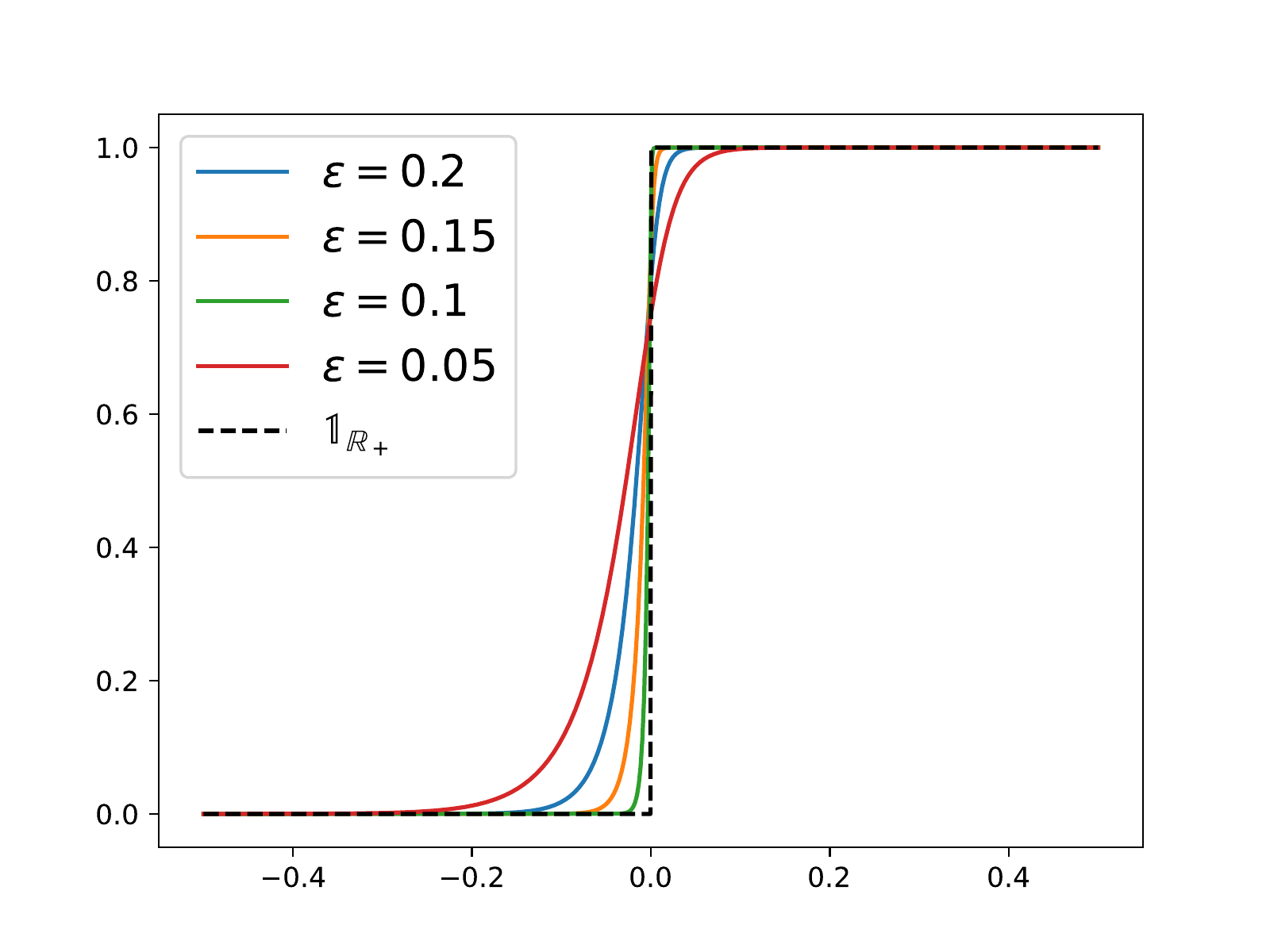}
		
	\end{center}
	\caption{Approximation of the indicator function with different values of $\varepsilon$ (zoom on the abscissa axis on the right)}
	\label{fig:ind}
\end{figure}

Computations have been performed with Bocop software on a standard laptop computer (with a Gauss II integration scheme, $600$ time steps and relative tolerance $10^{-10}$).
As one can see in Figure \ref{fig:comparSIR124} and Table \ref{table:results2} problems ${\cal P}_0$, ${\cal P}_1$, ${\cal P}_2$ present similar performances for peak values and computation time.
\begin{figure}[ht!]
	\begin{center}
		\includegraphics[width=0.45\textwidth]{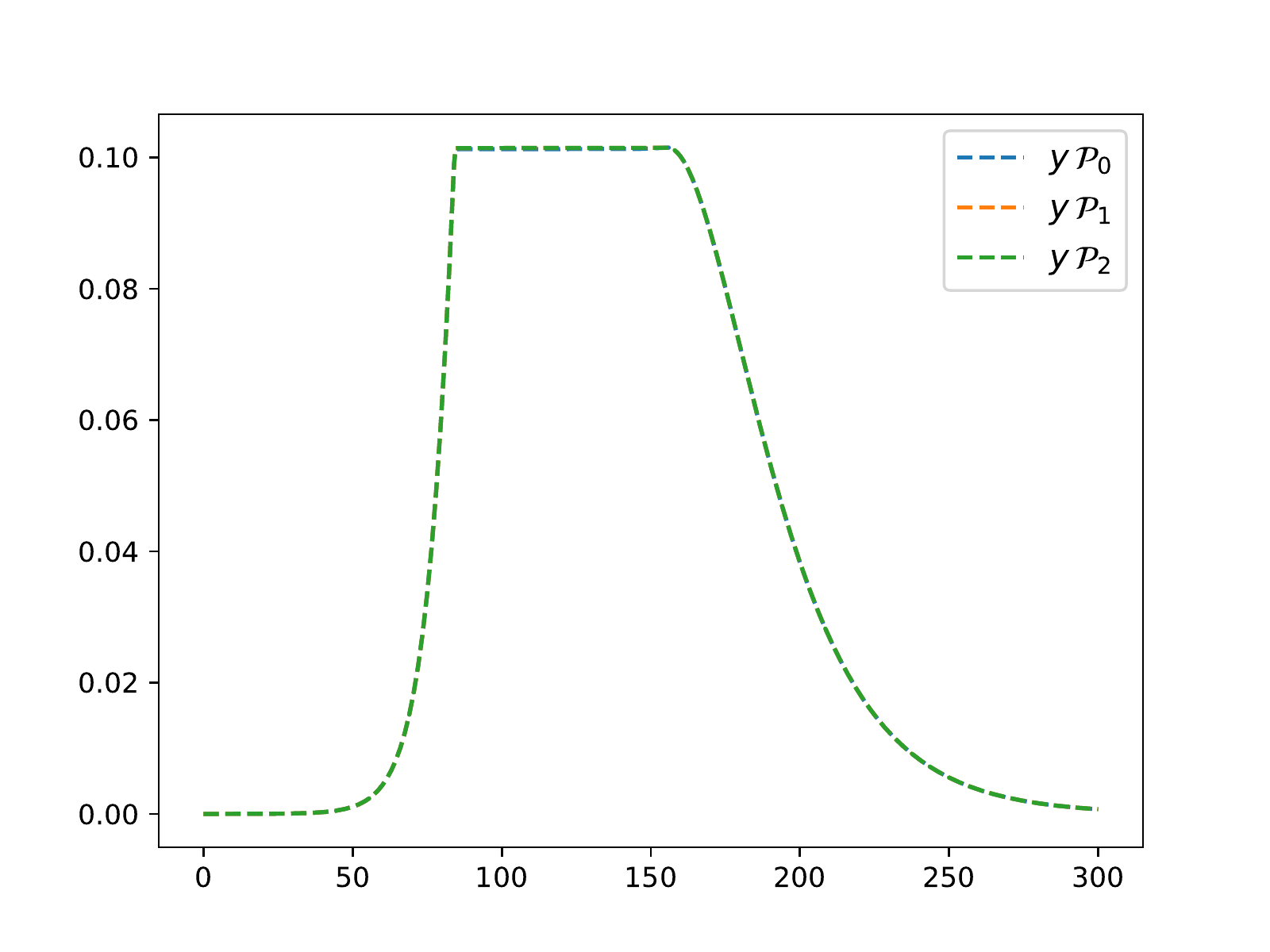}
		\includegraphics[width=0.45\textwidth]{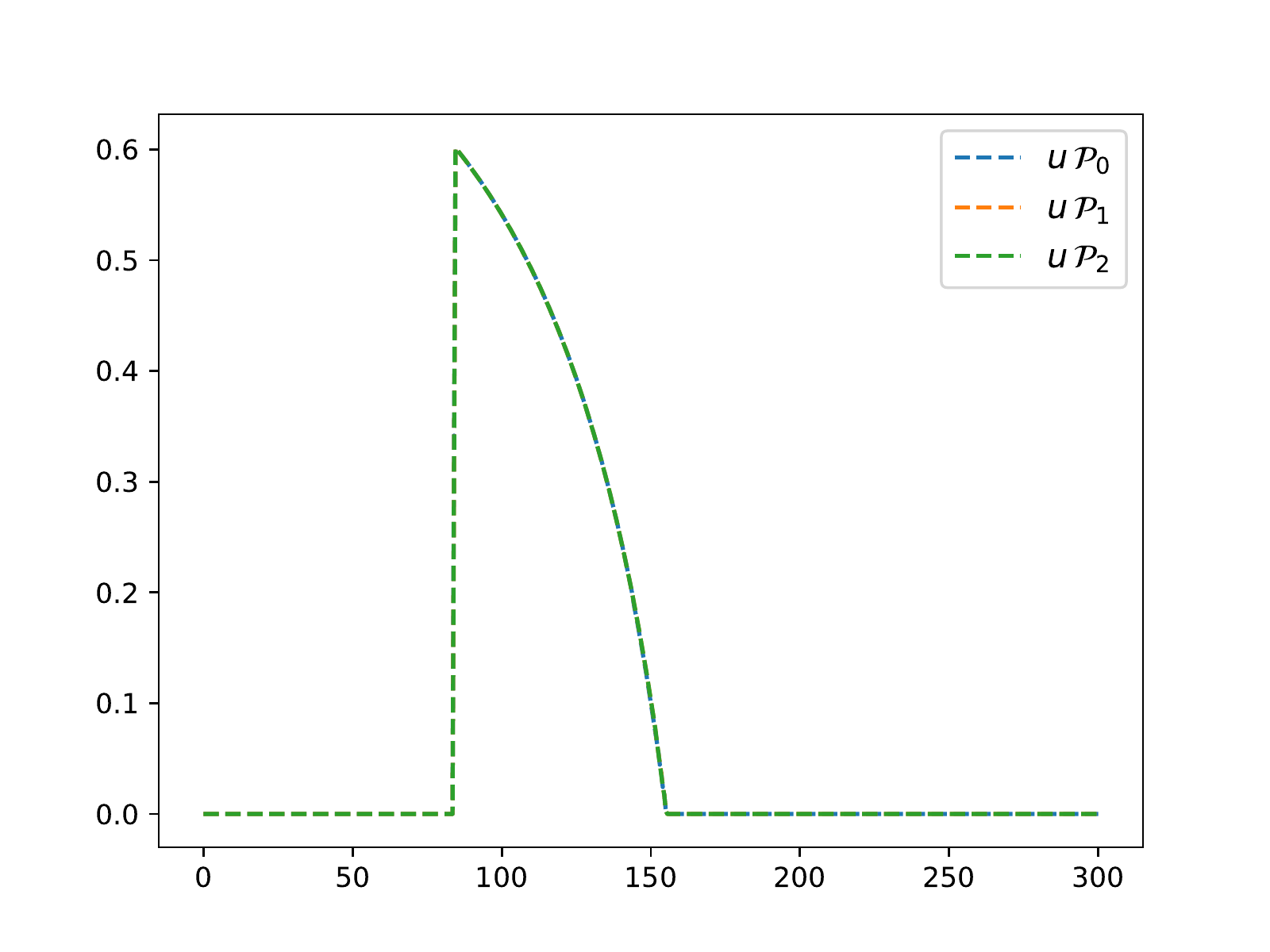}
		\includegraphics[width=0.45\textwidth]{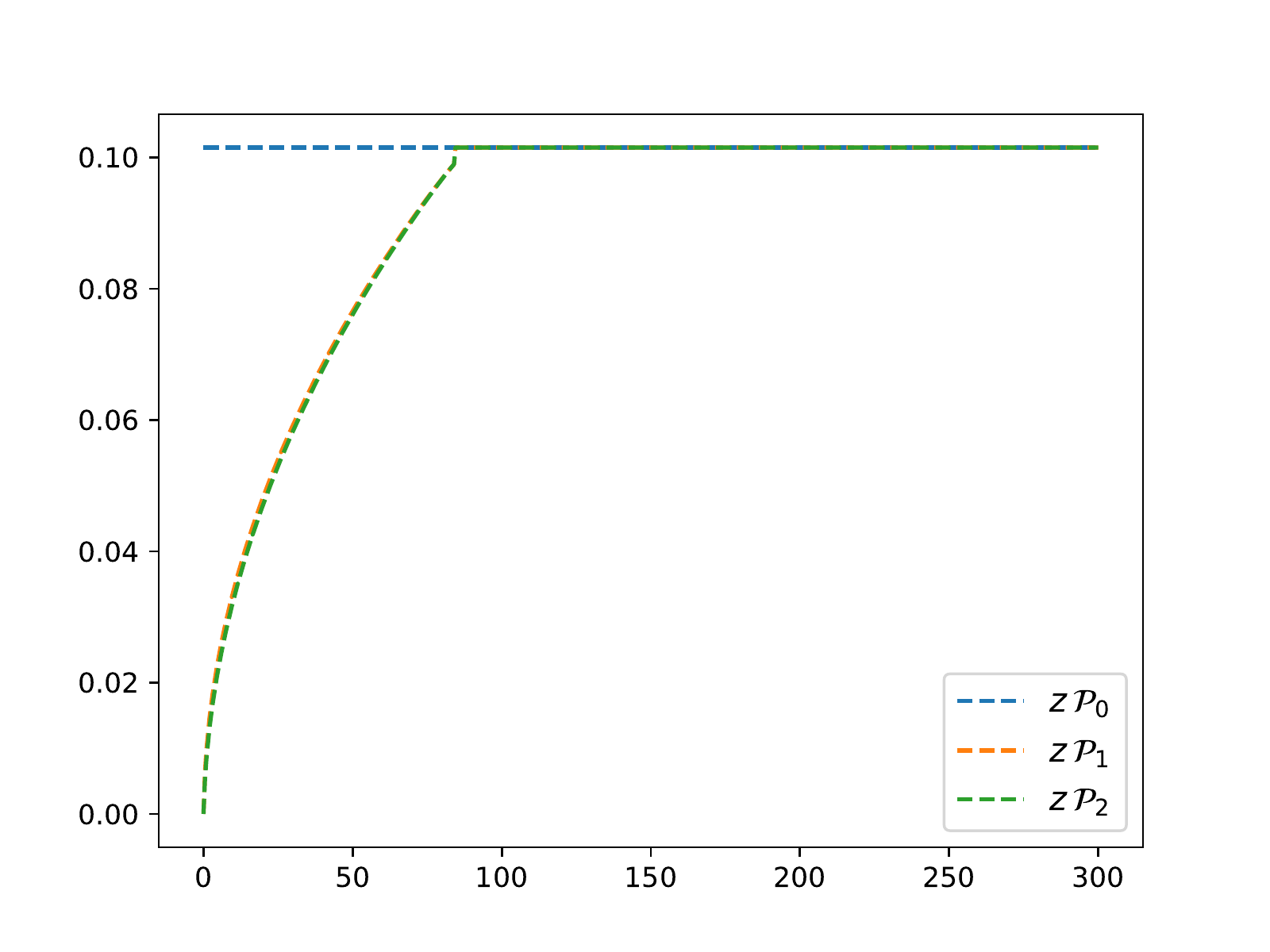}
		\includegraphics[width=0.45\textwidth]{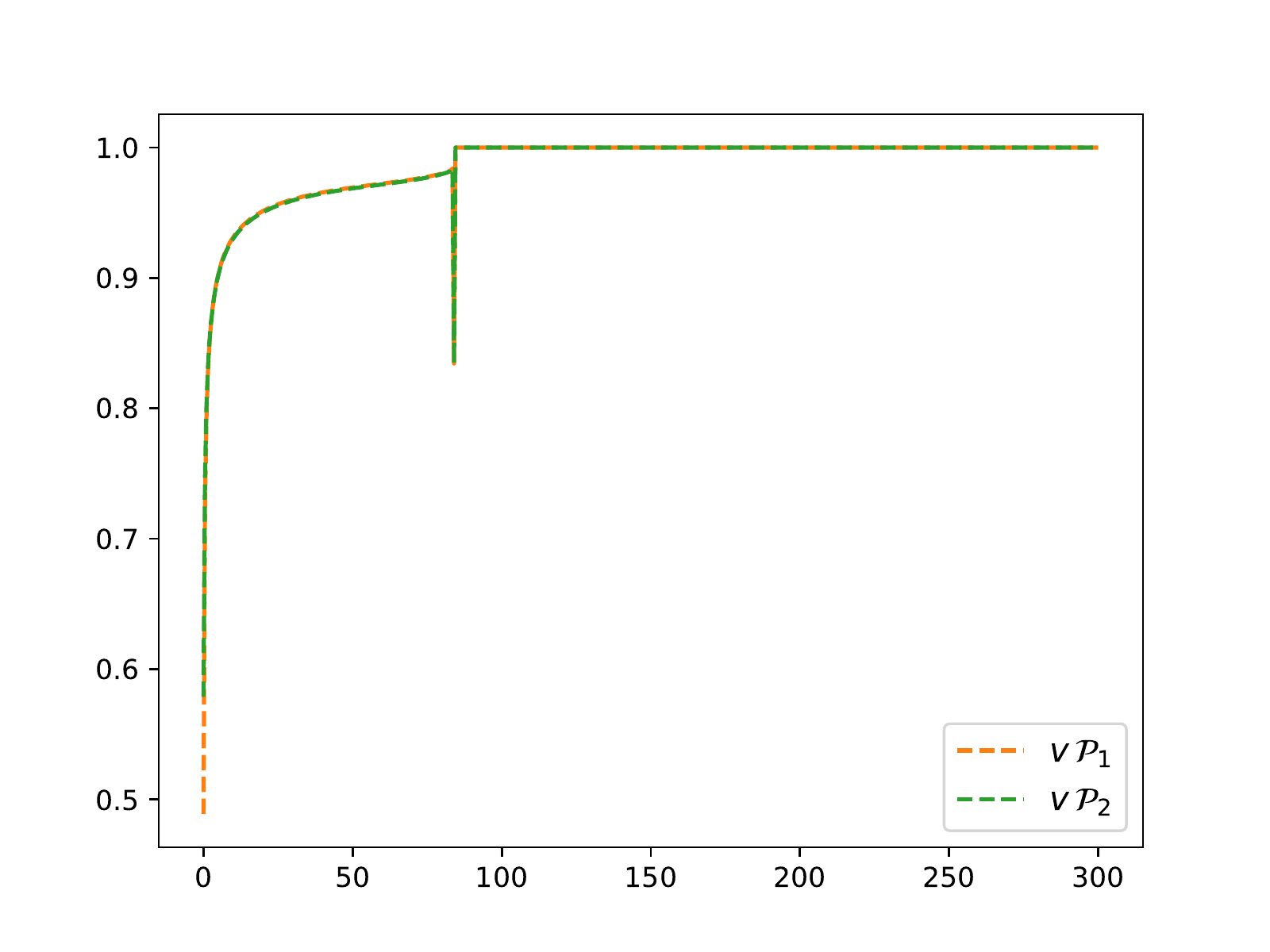}
	\end{center}
	\caption{Comparisons of numerical results for the methods ${\cal P}_0$, ${\cal P}_1$, ${\cal P}_2$}
	\label{fig:comparSIR124}
\end{figure}
\begin{table}[ht!]
	\begin{center}
		\begin{tabular}{c|c|c}
			problem & $\displaystyle \max_{t\in[0,T]} y(t)$ & computation time \\ \hline\hline
			${\cal P}_0$ &  $0.1015$ &  $10\, s$\\
			${\cal P}_1$ &  $0.1015$ &  $12\, s$\\
			${\cal P}_2$ &  $0.1015$ &  $13\, s$\\
			
		\end{tabular}
		\caption{Comparison of performances for problems ${\cal P}_0$, ${\cal P}_1$, ${\cal P}_2$}
		\label{table:results2}
	\end{center}
\end{table}
In Figure \ref{fig:comparSIR3} and Table \ref{table:results3}, the numerical solutions of $\mathcal{P}^{\theta}_3$ are illustrated for the values of $\alpha$ and $\lambda_2$ given in Table \ref{tab:alphalambda}.
\begin{table}[ht!]
	\begin{center}
		\begin{tabular}{c|c|c|c}
			$\varepsilon$ &$z(T)$& $\displaystyle \max_{t\in[0,T]} y(t)$ & computation time \\ \hline\hline
			$0.2$   & $0.0684$ & $0.1038$ &  $80\, s$\\
			$0.15$  & $0.0823$ & $0.1038$ &  $65\, s$\\
			$0.1$   & $0.0954$ & $0.1037$ &  $51\, s$\\
			$0.075$ & $0.0993$ & $0.1050$ &  $83\, s$\\
			$0.05$  & $0.1010$ & $0.1036$ &  $97\, s$
		\end{tabular}
		\caption{Comparison of performances for problem $\mathcal{P}^{\theta}_3$ }
		\label{table:results3}
	\end{center}
\end{table}
\begin{figure}[ht!]
	\begin{center}
		\includegraphics[width=0.3\textwidth]{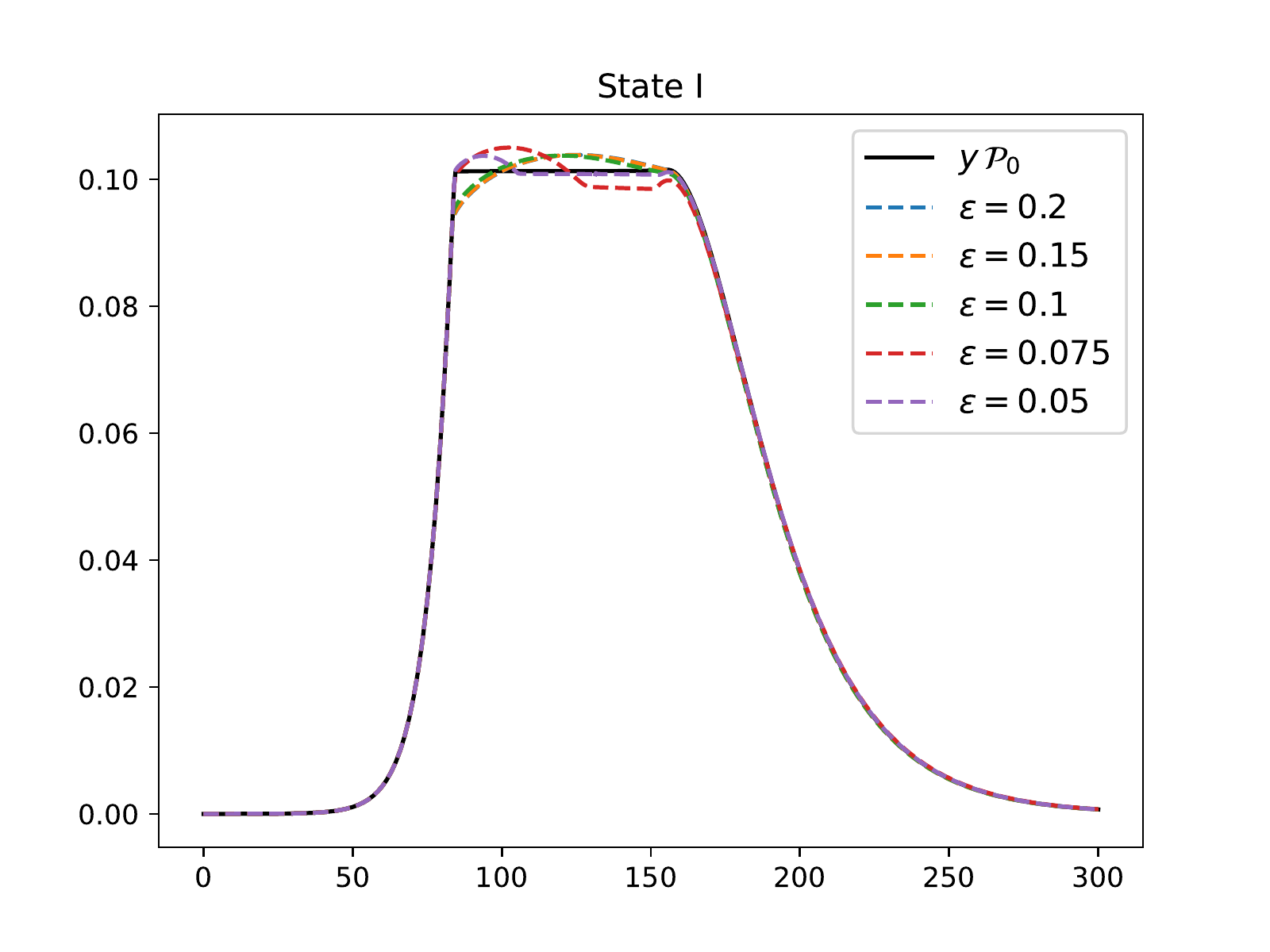}
		\includegraphics[width=0.3\textwidth]{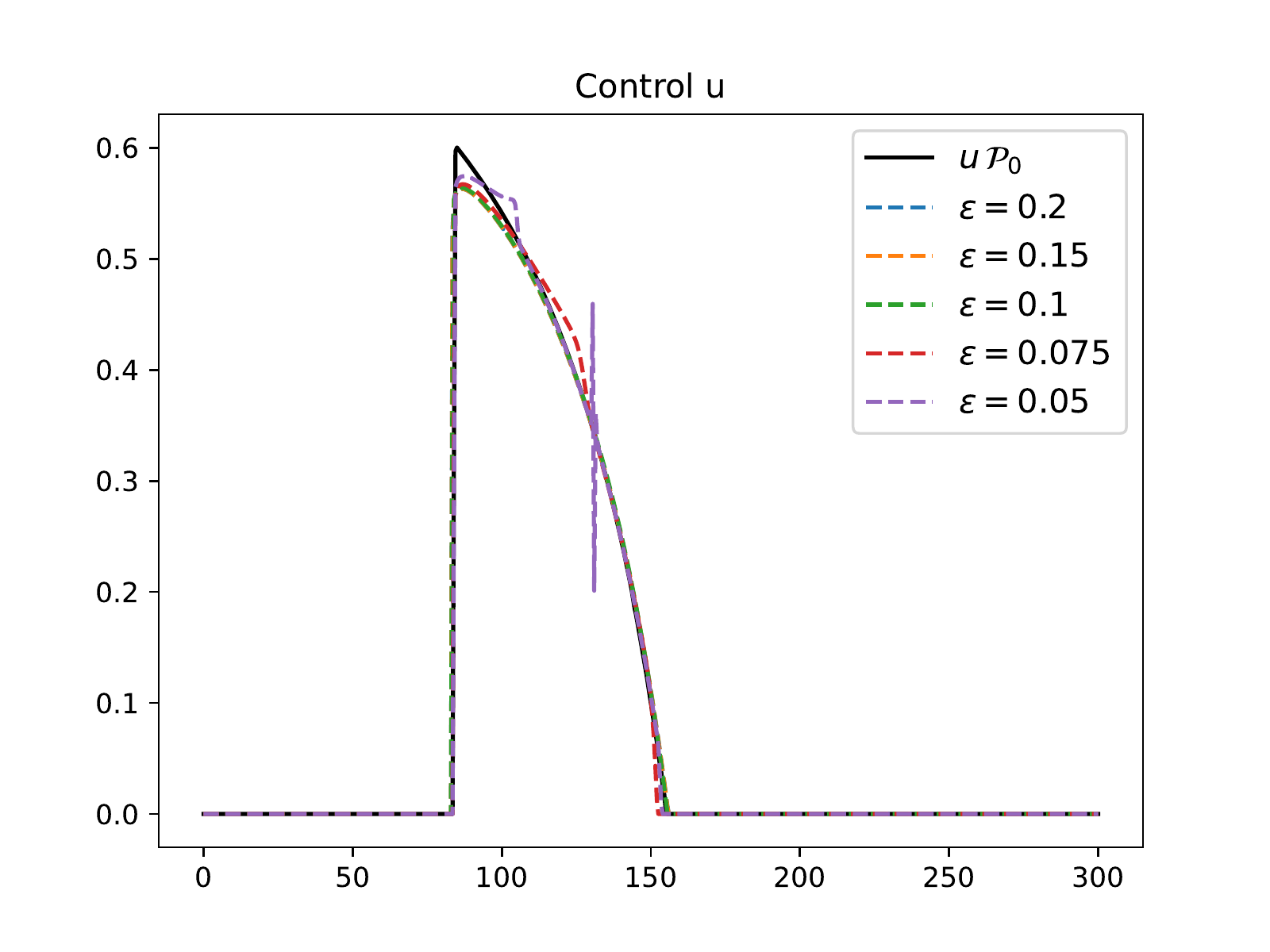}
		\includegraphics[width=0.3\textwidth]{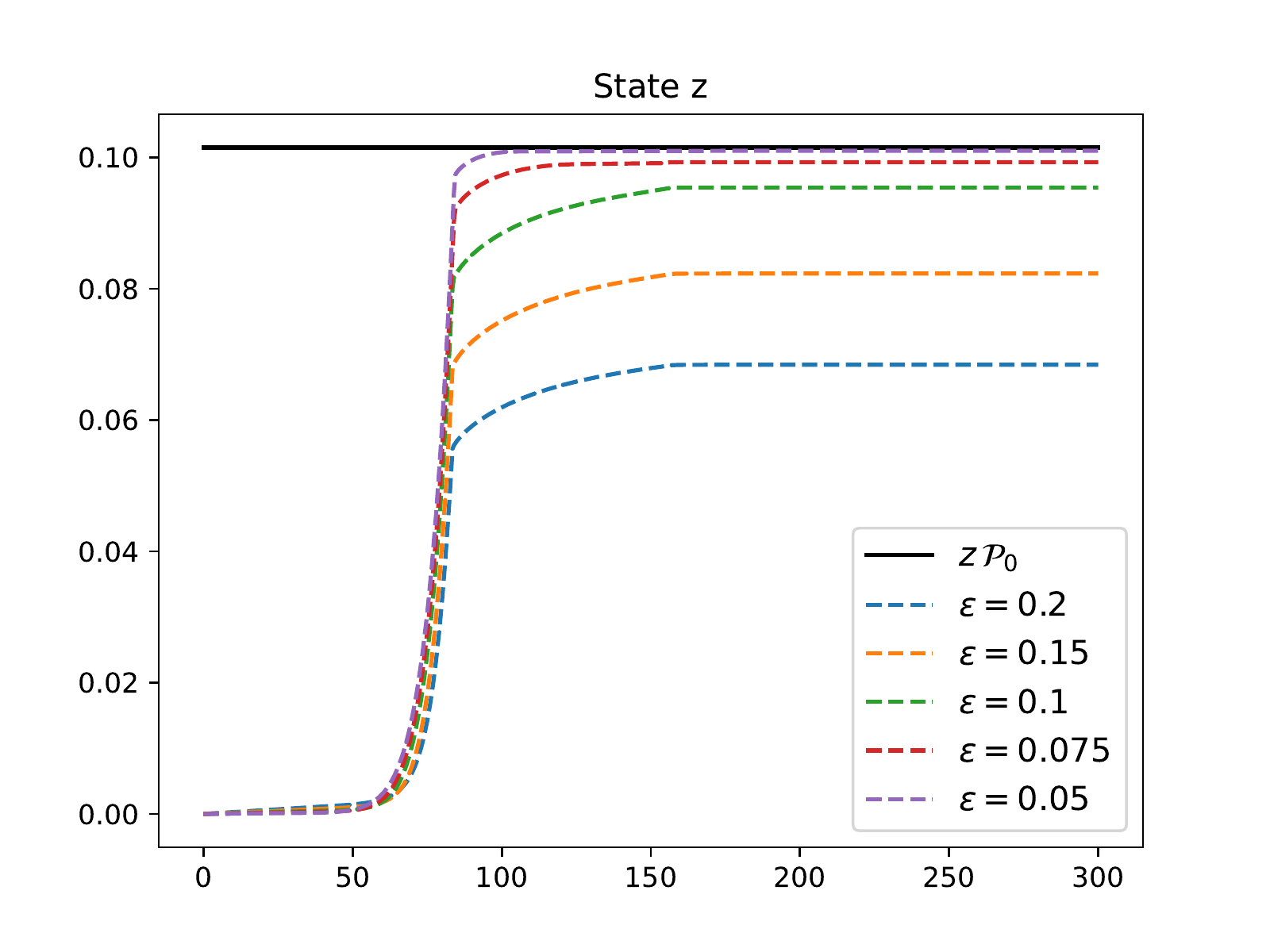}
	\end{center}
	\caption{Comparison of the numerical results for problem $\mathcal{P}^{\theta}_3$}
	\label{fig:comparSIR3}
\end{figure}
As expected, the numerical computation of the family of problems $\mathcal{P}^{\theta}_3$ provides an increasing sequence of approximation from below of the optimal value and thus complements the computation of problems $\mathcal{P}_0$,  $\mathcal{P}_1$ or $\mathcal{P}_2$. From Figures of Tables \ref{table:results2} and \ref{table:results3}, one can safely guarantee that the optimal value belongs to the interval $[0.1010,0.1015]$. However, the trajectories found for $\mathcal{P}^{\theta}_3$ are not as closed as the ones of problems $\mathcal{P}_0$,  $\mathcal{P}_1$ or $\mathcal{P}_2$. This can be explained by the fact that problems $\mathcal{P}^{\theta}_3$ are not subject to the constraint $z(t) \geq y(t)$ and thus provides trajectories for which $z(T)$ is indeed below $\max_t y(t)$. 


\bigskip

Finally, we have compared our approximation technique with the classical approximation of the $L_{\infty}$ criterion by $L_p$ norms
\begin{equation*}
	{\cal P}_{L_p} : \quad \inf_{u(\cdot) \in {\cal U}} ||y(t)||_{p}
\end{equation*}	
with the same direct method.
To speed up the convergence, we have used the Bocop facility which allows a batch mode which consists in initializing the search from a solution found for a former value of $p$, that have been taken  $p\in\lbrace 2,5,10,15\rbrace$ (see Figure \ref{fig:Lp}). Besides, to ensure convergence it was necessary take 1200 time step instead of 600 as in previous simulations.
\begin{figure}[ht!]
	\begin{center}
		\includegraphics[width=0.45\textwidth]{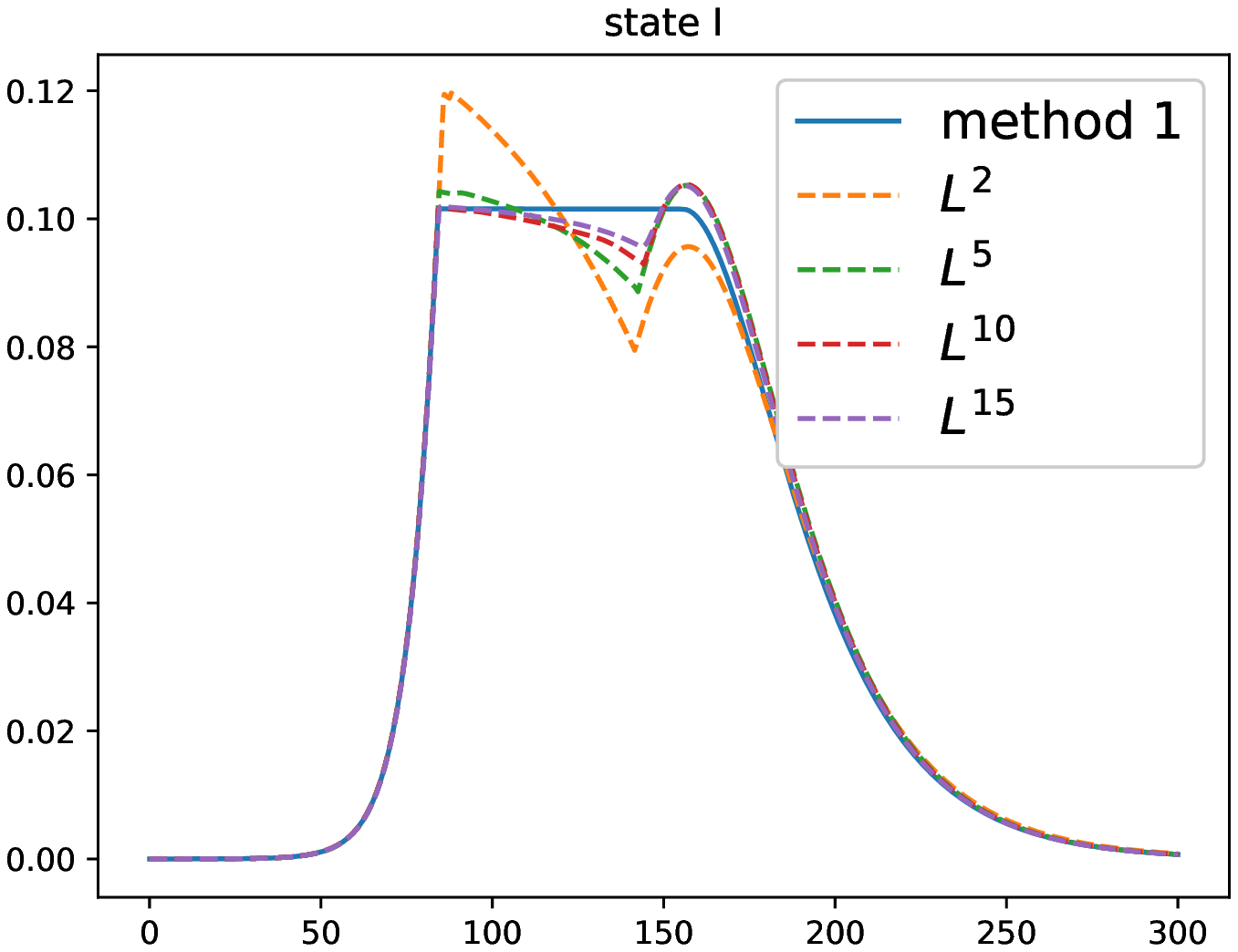}
		\includegraphics[width=0.45\textwidth]{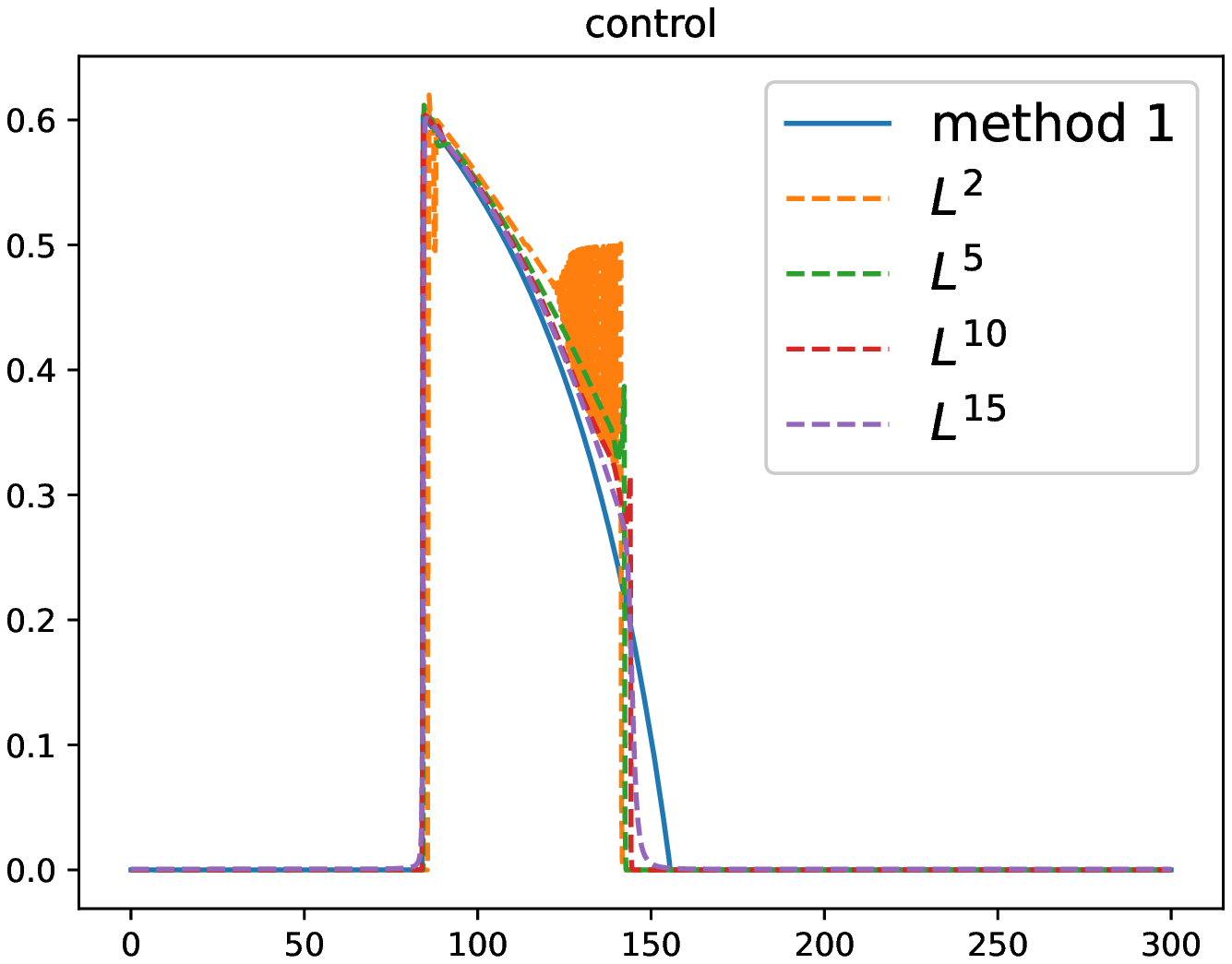}
	\end{center}
	\caption{Numerical solutions for problems ${\cal P}_{L_p}$}
	\label{fig:Lp}
\end{figure}
The total time of the process is $78 s$ after summing computation times given in Table \ref{table:resultsLp}.
\begin{table}[ht!]
	\begin{center}
		\begin{tabular}{c|c|c|c}
			$p$& $\displaystyle \max_{t\in[0,T]} y(t)$ & $||y(t)||_{p}$&computation time \\ \hline\hline
			$2$  & $0.119653$ & $1.0222$ &  $34\, s$\\
			$5$  & $0.105244$ & $0.2474$ &  $14\, s$\\
			$10$ & $0.105375$ & $0.15678$ &  $13\, s$\\
			$15$ & $0.105170$ & $0.13549$ &  $17\, s$\\
		\end{tabular}
		\caption{Comparison of the numerical results with the $L_p$ approximation}
		\label{table:resultsLp}
	\end{center}
\end{table}
However, one can see that the trajectory found for $p=15$ is quite far to give a peak value close from the other methods. Moreover, the same method for $p=15$ but initialized from the solution found for $p=2$ gives poor results for a computation time of $50 s$ (see Figure \ref{fig:Lp2}). We conclude that the $L_p$ approximation is not practically reliable for this kind of problems.
\begin{figure}[ht!]
	\begin{center}
		\includegraphics[width=0.45\textwidth]{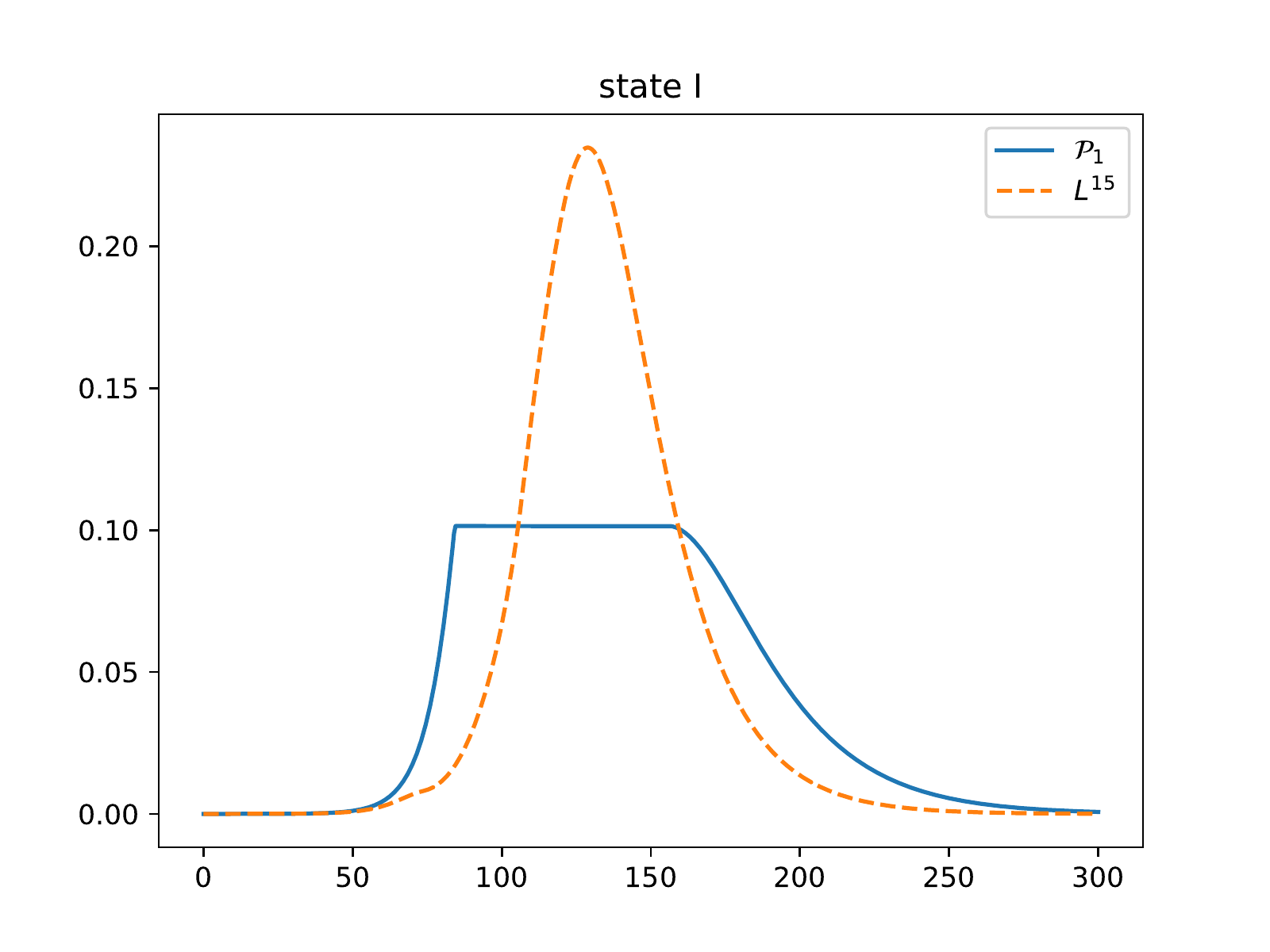}
		\includegraphics[width=0.45\textwidth]{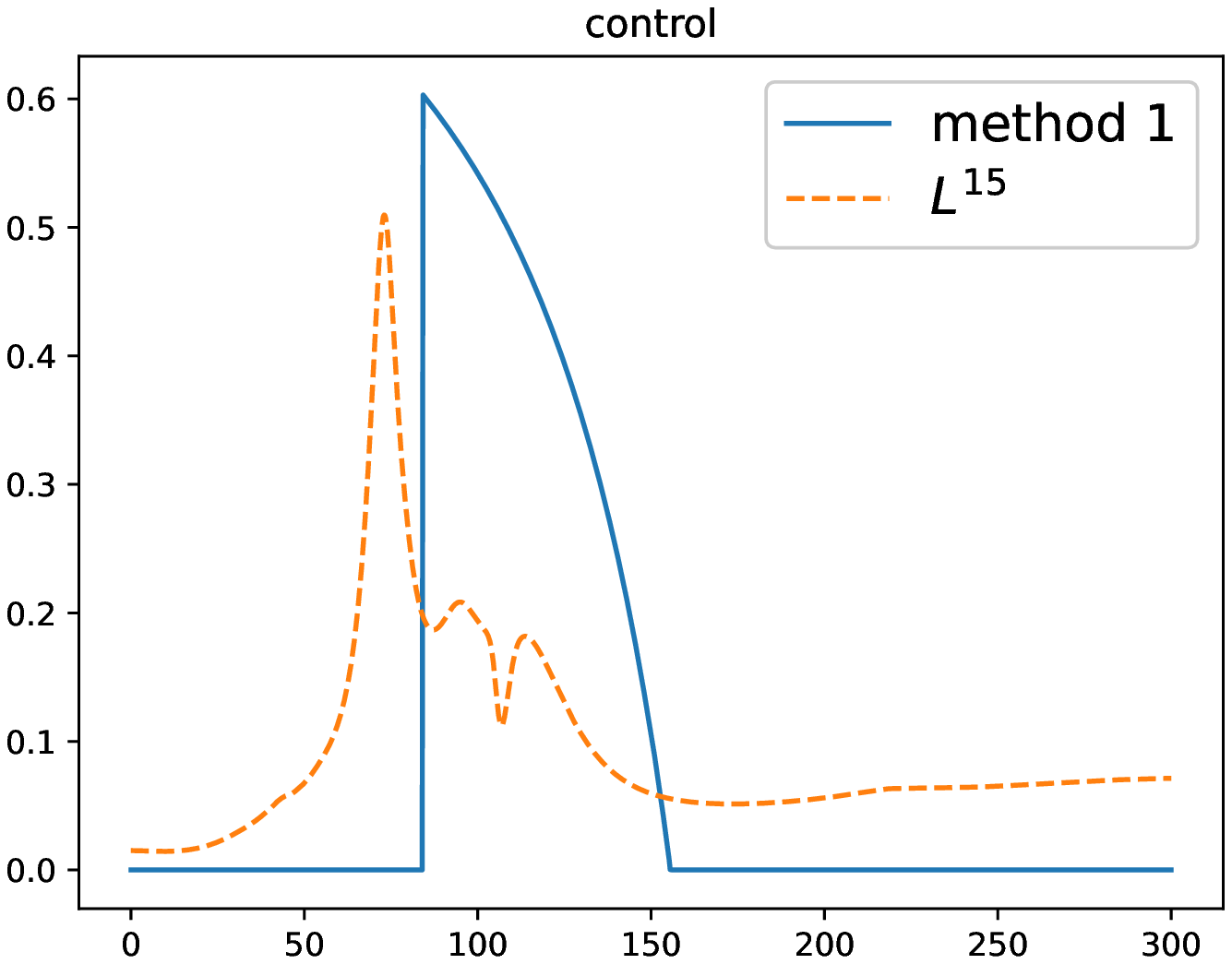}
	\end{center}
	\caption{Numerical solution for ${\cal P}_{L_{15}}$ without batch iteration (computation time $50 s$)}
	\label{fig:Lp2}
\end{figure}

\section{Discussion and conclusions}
\label{sec_conclusion}

In this work, we have presented different formulations of optimal control problems with maximum cost in terms of extended Mayer problems with fixed initial condition, and tested them numerically. We have proposed two classes of problems: one with state or mixed constraint suitable to direct methods, and another one without constraint but less regular and suitable to dynamical programming methods. Moreover, for this last class, we have proposed and approximation scheme with a sequence of regular Mayer problems without constraint, which turned out to give better results than approximations with the $L_p$ norms. Although this second approach requires larger computation time, it complements the first one providing approximations of the optimal value from above. 

Finally, we summarize some advantages and drawbacks of the different formulations for the use of numerical methods in Table \ref{table_comparison}.

\begin{table}[th!]
\begin{center}
    \begin{tabular}{l|c|c|c|c}
      Formulation   &  ${\cal P}_0$ & ${\cal P}_1$ or ${\cal P}_2$ & ${\cal P}_3$ & ${\cal P}_3^\theta$\\
      \hline\hline
    suitable to direct methods & yes  & yes & no & yes\\
    suitable to Hamilton-Jacobi-Bellman methods & no  & yes & yes & yes\\
    suitable to shooting methods without constraint & no  & no & no & yes\\
    provides approximations from below & no & no & no & yes
    \end{tabular}
    \caption{Comparison of the different formulations}
		\label{table_comparison}
	\end{center}
\end{table}

This first work puts in perspective the study of necessary optimality conditions for the maximum cost problems with the help of these formulations, which will be the matter of a future work.

\section*{Acknowledgments} The authors are grateful to Pierre Martinon for fruitful discussions and advices.
This work was partially supported by ANID-PFCHA/Doctorado Nacional/2018-21180348,
FONDECYT grant 1201982 and Centro de Modelamiento Matem\'atico (CMM), ACE210010 and FB210005, BASAL funds for center of excellence, all of them from ANID (Chile)


\begin{thebibliography}{99}
	
	\bibitem{Aubin}
	{\sc Aubin, J.-P.}, {\em Viability Theory}, Springer, 2009.
	
	\bibitem{Barron}
	{\sc Barron, E. N.}, {\em The Pontryagin maximum principle for minimax problems of optimal
control}, Nonlinear Anal. 15, 1155--1165, 1990.
	
	\bibitem{BarronIshii}
	{\sc Barron, E.N. and Ishii, H.}
	{\em The Bellman equation for minimizing the maximum cost},
	Nonlinear Analysis: Theory, Methods \& Applications,
13(9), 1067--1090, 1989.
	
	\bibitem{BarronJensenLiu}
	{\sc Barron, E. N.; Jensen, R. R. and Liu, W.}, {\em The $L^\infty$ control problem with continuous control
 functions}, Nonlinear Anal. 32, 1--14, 1998.
 
 \bibitem{Clarke}
 {\sc Clarke, F.}, {\em Optimization and Nonsmooth Analysis}, SIAM Classics in Applied Mathematics, 1990.
	
	\bibitem{DiMarcoGonzalez1996}
	{\sc Di Marco, A. and Gonzalez, R. L. V.}, {\em A numerical procedure for minimizing the maximum
 cost}, in: System Modelling and Optimization (Prague, 1995), Chapman \& Hall, London,
285–291, 1996.

	\bibitem{DiMarcoGonzalez1999}
	{\sc Di Marco, A. and Gonzalez, R. L. V.}, {\em Minimax optimal control problems. Numerical analysis of the finite horizon case}, ESAIM: Mathematical Modelling and Numerical Analysis 33(1), 23--54, 1999.

	\bibitem{GianattiAragoneLolitoParente}
	{\sc Gianatti, J.; Aragone, L.; Lotito, P. and Parente, L.}, {\em Solving minimax control problems via nonsmooth optimization}. Operations Research Letters 44., 680--686, 2016.

	\bibitem{GonzalezAragone}
	{\sc Gonzalez, R.L.V. and Aragone, L.}, {\em A Bellman's equation for minimizing the maximum cost.} Indian Journal of Pure \& Applied Mathematics 31(12), 1621--1632, 2000.
	
	\bibitem{MorganPeet2020}
	{\sc Morgan,J. and  Peet, M.}
	{\em Extensions of the Dynamic Programming Framework: Battery Scheduling, Demand Charges, and Renewable Integration}, IEEE Transactions on Automatic Control, 66(4), pp. 1602-1617, 2021.
	
	\bibitem{MorganPeet2021}
	{\sc Morgan,J. and  Peet, M.}
	{\em A generalization of Bellman’s equation with application to path planning, obstacle avoidance and invariant set estimation}.
	Automatica,
127, n. 109510,
2021.
	
	\bibitem{Morris} {\sc Morris, D. , Rossine, F., Plotkin, J., \& Levin, S. }, {\em Optimal, near-optimal, and robust epidemic control.} arXiv preprint Communications Physics 4.1, 1-8, 2021.	
	
	\bibitem{Tao}
	{\sc Tao, T.} {\em An introduction to measure theory}. Graduate Studies in Mathematics 126, AMS Society, 2011.
	
	\bibitem{Walter}
	{\sc Walter, W.}, {\em Ordinary Differential Equations}, Springer New-York, 1998.
	
	\bibitem{SIR}
	{\sc Weiss, H.}, {\em The SIR model and the foundations of public health.}
	MATerials MATem\`atics, 2013(3), 2013.
		

	
	\end{thebibliography}
\end{document}